\documentclass[leqno]{article}  
\usepackage{amsmath,amsthm,amssymb,graphics,enumerate}
\usepackage{hyperref}
\usepackage[all]{xy}
\usepackage[letterpaper,margin=1.1in]{geometry}

\date{}

\theoremstyle{plain}
\newtheorem{Theorem}{Theorem}[section]
\newtheorem*{MainTheorem}{Main Theorem}
\newtheorem{Lemma}[Theorem]{Lemma}
\newtheorem{Corollary}[Theorem]{Corollary}
\newtheorem{Proposition}[Theorem]{Proposition}
\newtheorem{Theorem-Definition}[Theorem]{Theorem-Definition}
\newtheorem{Lemma-Definition}[Theorem]{Lemma-Definition}
\newtheorem*{theorem*}{Theorem}
\newtheorem*{satellite*}{Satellite to the Main Theorem}
\newtheorem{Claim}{Claim}

\theoremstyle{definition}

\newtheorem{Observation}[Theorem]{Observation}
\newtheorem*{Observation*}{Observation}

\newtheorem{Example}[Theorem]{Example}

\newtheorem*{Remark*}{Remark}

\def\edge{\ar@{-}}
\def\drbl{\save+<0ex,-2ex> \drop{\bullet} \restore}
\def\dashedge{\ar@{--}}
\def\dropup#1{\save+<0ex,4ex> \drop{#1} \restore}
\def\dropleft#1{\save+<-7ex,0ex> \drop{#1} \restore}

\def\la{{\Lambda}}
\def\lamod{\Lambda\mbox{\rm-mod}}

\def\bdel{{\mathbf{\partial}}}
\def\S{{\sigma}}

\def\AA{{\mathbb A}}
\def\CC{{\mathbb C}}

\def\SS{{\mathbb S}}
\def\ZZ{{\mathbb Z}}
\def\NN{{\mathbb N}}
\def\QQ{{\mathbb Q}}

\def\C{{\mathcal C}}

\def\U{{\mathfrak U}}

\def\bb{\mathbf {b}}

\def\bB{\mathbf {B}}

\def\bE{\mathbf {E}}

\def\bd{\mathbf {d}}

\def\bp{\mathbf {p}}
\def\bq{\mathbf {q}}

\def\Phat{\widehat{P}}

\def\SShat{\widehat{\SS}}

\def\lahat{\widehat{\la}}

\def\Dhat{\widehat{D}}

\def\Ghat{\widehat{G}}

\def\Khat{\widehat{K}}

\def\Phat{\widehat{P}}

\def\SStilde{\widetilde{\SS}}

\def\ptilde{\widetilde{p}}

\def\SStilde{\widetilde{\SS}}

\def\soc{\operatorname{soc}}
\def\radsocbd{\operatorname{\mathbf{rad-soc(\bd)}}}

\def\radsoc{\operatorname{\mathbf{rad-soc}}}

\def\spec{\operatorname{Spec}}

\def\ann{\operatorname{ann}}

\def\ann{\operatorname{ann}}

\def\start{\operatorname{start}}

\def\Im{\operatorname{Im}}
\def\rep{\operatorname{\mathbf {Rep}}}
\def\Seq{\operatorname{\mathbf {Seq}}}

\def\Ext{\operatorname{Ext}}

\def\End{\operatorname{End}}

\def\udim{\operatorname{\underline{dim}}}

\def\modlad{\operatorname{\mathbf{Rep}}_{\mathbf{d}}(\Lambda)}

\def\Mod{\operatorname{\mathbf{Rep}}}

\def\laySS{\operatorname{\mathbf{Rep}} \SS}

\def\laySStilde{\operatorname{\mathbf{Rep}} \SStilde}

\def\Hom{\operatorname{Hom}}

\def\Gal{\operatorname{Gal}}

\def\start{\operatorname{start}}
\def\Schu{\operatorname{Schu}}

\def\Schu{\operatorname{\Schu}}

\def\GRASS{\operatorname{GRASS}}

\def\grassbd{\GRASS_{\mathbf d}(\Lambda)}

\def\start{\operatorname{start}}
\def\Schu{\operatorname{Schu}}
\def\spec{\operatorname{Spec}}

\title{\textbf{Irreducible components of varieties of representations: The acyclic case}}

\author{
B. Huisgen-Zimmermann
\hspace{1ex}and I. Shipman
}

\begin{document}
\maketitle

\begin{abstract}  The goals of this article are as follows: (1) To determine the irreducible components of the affine varieties $\modlad$ parametrizing the representations with dimension vector $\bd$, where $\la$ traces a major class of finite dimensional algebras; (2) To generically describe the representations encoded by the components.  The target class consists of those truncated path algebras $\la$ over an algebraically closed field $K$ which are based on a quiver $Q$ without oriented cycles.  The main result characterizes the irreducible components of $\modlad$ in representation-theoretic terms and provides a means of listing them from quiver and Loewy length of $\la$.  Combined with existing theory, this classification moreover yields an array of generic features of the modules parametrized by the components, such as generic minimal projective presentations, generic sub- and quotient modules, etc.  Our second principal result pins down the generic socle series of the modules in the components; it does so for more general $\la$, in fact.

The information on truncated path algebras of acyclic quivers supplements the theory available in the special case where $\la = KQ$, filling in generic data on the $\bd$-dimensional representations of $Q$ with any fixed Loewy length.
\end{abstract}
\vskip1.0truein
 
Corresponding author: Birge Huisgen-Zimmermannn,   Department of Mathematics, University of California, Santa Barbara, CA 93106. Phone: 1-805-893-2915. Fax: 1-805-893-2385. Email: birge@math.ucsb.edu.
\smallskip

Additional author: Ian Shipman, Department of Mathematics, Harvard University, Cambridge, MA 02138. Email: ian.shipman@gmail.com. 
\bigskip

\bigskip


\vfill\eject

\section{Introduction and main results}\label{s-intro}

Our purpose is to promote the development and application of strategies for organizing the representation theory of basic finite dimensional algebras $\la$ on a generic level.  This approach  to representation theory was initiated by Kac and Schofield in the hereditary case, that is, when $\la = KQ$ is a path algebra; see \cite{KacI, KacII} and \cite{Scho}.  While it is obvious that, for hereditary $\la$, the parametrizing varieties $\modlad$ of the modules with fixed dimension vector $\bd$ are full affine spaces, on moving beyond the hereditary case, the initial task is to identify the irreducible components of $\modlad$ in a representation-theoretically useful format.  Supporting theory was developed, and numerous special cases of this task were resolved, in \cite{DoFl, Mor, CBS,  BaSch, GeiSchI, GeiSchII, Schro, RiRuSm, CW, BHT} for instance.  
We refer to  \cite{irredcompI} for a somewhat more detailed overview, but mention that, up to \cite{BCH} and \cite{irredcompI}, full solutions were available only for some classes of tame algebras with fully classified finite dimensional indecomposable representations.  In essence, we are dealing with varieties of matrices satisfying certain relations, a problem of interest beyond the representation theory of algebras; see, e.g., \cite{EiSa, Ger, Gur}.

Our main result, stated below, addresses \textit{truncated path algebras}, that is, algebras of the form $\la = KQ/ \langle \text{all paths of length\ } L+1\rangle$ for a quiver $Q$ and a fixed positive integer $L$.  Clearly, this class includes the hereditary algebras and those with vanishing radical square.  A brief discussion of the  prominent place held by the truncated path algebras in this connection can be found in \cite{irredcompI}.   

The starting line of our present investigation is provided by the facts that, for any truncated path algebra $\la$ and dimension vector $\bd$, all subvarieties of the form $\laySS$ of $\modlad$ are irreducible, and the irreducible components of $\modlad$ are among the closures $\overline{\laySS}$ (see \cite{BHT}).  Here $\laySS$ consists of those points in $\modlad$ which encode modules with a given radical layering $\SS$.  Recall that the radical layering of a $\la$-module $M$ is the sequence $\SS(M) = \bigl(J^l M/ J^{l+1}M\bigr)_{0 \le l \le L}$, where $J$ is the Jacobson radical of $\la$; we assume that $J^{L+1} = 0$.  The socle layering $\SS^*(M)$ is defined dually.  Both of the isomorphism invariants $\SS(M)$ and $\SS^*(M)$ of $M$ are \textit{semisimple sequences} having the same dimension vector as $M$, i.e., they are sequences of the form $\SS = (\SS_0, \SS_1, \dots, \SS_L)$ whose entries $\SS_l$ are semisimple modules such that $\udim \SS : = \sum_{0 \le l \le L} \udim \SS_l$ equals $\udim M$.  In defining $\laySS$ as above, we prioritize radical layerings and identify isomorphic semisimple modules.  

In the truncated case, our task is thus reduced to pinning down the semisimple sequences $\SS$ for which $\overline{\laySS}$ is \emph{maximal} among the irreducible subvarieties of $\modlad$.  To sift them out of the set $\Seq(\bd)$ of all semisimple sequences with dimension vector $\bd$, our chief tool is the following upper semicontinuous map
$$\Theta:  \modlad \rightarrow \Seq(\bd) \times \Seq(\bd), \ \ x \mapsto \bigl(\SS(M_x), \SS^*(M_x)\bigr),$$
where $\Seq(\bd) \times \Seq(\bd)$ is partially ordered by the componentwise dominance order  --  see under Conventions below --  and $M_x$ is the $\la$-module corresponding to $x$.  The image of $\Theta$ is denoted by $\radsocbd$.  Due to semicontinuity of $\Theta$, each minimal pair $(\SS, \SS^*)$ in $\radsocbd$ gives rise to an irreducible component $\overline{\laySS}$ of $\modlad$.  Yet,
 for arbitrary truncated $\la$, the map $\Theta$ has blind spots. Indeed, \cite[Example 4.8]{irredcompI} shows that not all irreducible components of $\modlad$ correspond to minimal elements of $\radsoc(\bd)$ in general.  In important special cases, $\Theta$ \emph{does} detect all irreducible components however.  By \cite{irredcompI}, this is true when $\la$ is \emph{local} truncated; in that situation, the minimal pairs in $\radsoc(\bd)$ may be recognized by mere inspection of the sequence of consecutive dimensions  $\bigl(\dim \SS_l \bigr)_{0 \le l \le L}$.  Here we tackle  --  with different methods  --  the quivers located at the opposite end of the spectrum, namely the acyclic ones.     

\begin{MainTheorem}\hypertarget{thm-main}{} Let $\la$ be a truncated path algebra based on an acyclic quiver $Q$, and let $\bd$ be a dimension vector.  Then the following conditions are equivalent for any semisimple sequence $\SS$ with $\udim \SS = \bd$:
\begin{enumerate}
\item[\rm(1)] The closure of $\laySS$ is an irreducible component of $\modlad$. 
\item[\rm(2)] $\SS$ occurs as the first entry of a minimal pair $(\SS, \SS^*)$ in $\radsocbd$.
\end{enumerate}
The situation is symmetric in $\SS$ and $\SS^*$:  Whenever $(\SS, \SS^*)$ is a minimal element of $\radsocbd$, then $\SS^*$ is the generic socle layering of ${\laySS}$; conversely, $\SS^*$ determines the generic radical layering $\SS$ of the modules with socle layering $\SS^*$.
\end{MainTheorem}

The Main Theorem is proved in Section \ref{s-proof-main-thm}.  It translates into an analogous statement for the projective parametrizing varieties $\grassbd$ of Grassmann type (cf.~\cite[Proposition 2.1]{irredcompI}).  In this connection we also point to \cite{CFR}, where the irreducible components of quiver Grassmannians over Dynkin quivers are addressed.

Since each irreducible component of $\modlad$ is of the form $\overline{\laySS}$, the theorem allows for an explicit classification of the components of $\modlad$ from $Q$, $L$ and $\bd$.  To expedite the process of filtering the minimal elements out of $\radsoc(\bd)$, we wish to minimize the number of comparisons required:  For any $\SS \in \Seq(\bd)$, there is a unique minimal socle layering attained on the modules in $\laySS$, namely the generic one.  This layering is readily obtained from $\SS$ as follows.  To our knowledge, a description of generic socles had previously not even been available in the hereditary case.

\begin{theorem*}\label{supplement}  {\rm(See Theorem \ref{mainsoc} for a full statement.)} Let $\la$ be a truncated path algebra based on an arbitrary quiver $Q$ $($not necessarily acyclic$)$ and  $\mathbf{A}$ the transpose of the adjacency matrix of $Q$.  For any semisimple sequence $\SS = (\SS_0, \dots, \SS_L)$, the generic socle $\SS^*_0$ of the modules in $\overline{\laySS}$ is given by the dimension vector 
$$ \sup\, \biggl\{\, \sum_{L-m \le l \le L} \biggl( \udim \SS_l - \udim \SS_{l+1} \cdot \mathbf{A} \biggr) \biggm| 0 \le m \le L \biggr\};$$
here $\SS_{L+1} = 0$, and the supremum is taken with respect to the componentwise partial order on $\ZZ^{|Q_0|}$.

The higher entries of the generic socle layering $\SS^* = (\SS^*_0, \dots, \SS^*_L)$ of the modules in $\overline{\laySS}$ may be derived recursively from $\SS^*_0$.    
\end{theorem*}
 
In the special case of a hereditary algebra $\la = KQ$, we first determine the unique generic radical layering $\SS$ for $\modlad$ directly from $\bd$ (Corollary \ref{cor-generic-tops}).  Then we exploit this information towards a broader generic analysis of the $\bd$-dimensional representations of $Q$.  By its nature, Kac's and Schofield's seminal work on path algebras limits its focus to the modules of maximal Loewy length $L(\bd)$.  Our results on truncations $KQ/I$ provide an extension  towards a generic understanding of the $\bd$-dimensional $KQ$-modules of arbitrary Loewy length $< L(\bd)$. Indeed, excising the open subvariety of $\Mod_\bd(KQ)$ which encodes the modules of Loewy length $L(\bd)$ leaves us with a copy of the variety 
$\Mod _\bd \bigl(KQ/ \langle L(\bd) -1 \rangle \bigr)$, where $\langle m \rangle$ denotes the ideal of $KQ$ generated by all paths of length $m$.  Iteration shifts the generic focus to $\Mod _\bd \bigl(KQ/ \langle m \rangle \bigr)$ for successively smaller $m$. The results we sketched above thus provide access to the irreducible components of the locally closed subvariety of $\Mod_\bd(KQ)$ representing the $KQ$-modules of any Loewy length $m < L(\bd)$, as well as to generic properties of the corresponding representations.  In exploring the components of the truncated path algebras $\Mod _\bd \bigl(KQ/ \langle m \rangle \bigr)$ for decreasing values of $m$, we are thus, in effect, targeting the generic behavior of successive classes of $\bd$-dimensional representations of $Q$, each step moving us to an irreducible subvariety of $\Mod_\bd(KQ)$ that had been blended out in the previous steps (cf.~Example \ref{ex-nonhered-1}). 

Overview:  In Section \ref{s-foundational}, we assemble foundational material on generic modules; this section addresses arbitrary basic finite dimensional algebras, supplementing the general theory developed in \cite{BHT}.  It is only in Section \ref{s-truncated-path-algebras} that we specialize to truncated path algebras (based on arbitrary quivers), in order to (a) prepare tools for the proof of the main result, and (b) provide the theoretical means for  its effective application.  In particular, it is shown how, in the truncated case, the generic radical layering of an irreducible component provides access to its generic modules, generic submodules and quotients (Corollaries \ref{cor-subs-quots-gen-mods}, \ref{socquots}), as well as to the generic socle layering (Theorem \ref{mainsoc}).  In Section \ref{s-hereditary}, we narrow the focus from the truncated to the hereditary scenario, and in Section \ref{s-nonhered}, we illustrate the theory with non-hereditary examples. Section \ref{s-proof-main-thm}, finally, contains a proof of the \hyperlink{thm-main}{Main Theorem}.  
\bigskip

\textit{Conventions and prerequisites}:  For our technique of graphing $\la$-modules, we cite \cite[Definition 3.9 and subsequent examples]{BHT}.  Throughout, $\la$ is a basic finite dimensional algebra over an algebraically closed field $K$, and $\lamod$ (resp.~$\text{\rm{mod-}}\la$) is the category of finitely generated left (resp.~right) $\la$-modules. By $J$ we denote the Jacobson radical of $\la$; say $J^{L+1} = 0$.  Without loss of generality, we assume that $\la = KQ/ I$ for some quiver $Q$ and admissible ideal $I \subseteq KQ$.  Products of paths in $KQ$ are to be read from right to left. 

The set $Q_0 = \{e_1, \dots, e_n\}$ of vertices of $Q$ will be identified with a full set of primitive idempotents of $\la$.  Hence, the simple left $\la$-modules are $S_i = \la e_i/ J e_i$, $1 \le i \le n$, up to isomorphism.  Unless we want to distinguish among different embeddings, we systematically identify isomorphic semisimple modules; in other words, we identify finitely generated semisimples with their dimension vectors.  

Let $\SS$ be a semisimple sequence, that is, a sequence of the form $\SS = (\SS_0, \SS_1, \dots, \SS_L)$ such that each $\SS_l$ is a semisimple module, and set $\udim \SS = \sum_{0 \le l \le L} \udim \SS_l$.  When $\SS_l = 0$ for all $l \ge m+1$, we will also write $\SS$ in the clipped form $(\SS_0, \dots, \SS_m)$.  In light of the mentioned identifications, the collection $\Seq(\bd)$ of semisimple sequences with dimension vector $\bd$ is a finite set.  It is endowed with the following partial order, dubbed the \textit{dominance order} in \cite{hier}:  
$$\SS \le \SS' \ \ \ \iff \ \ \ \bigoplus_{0 \le j \le l} \SS_j \subseteq \bigoplus_{0 \le j \le l} \SS'_j \ \ \text{for}\  \ l \in \{0, 1, \dots, L\}.$$
Upper semicontinuity of the  map $\Theta:  \modlad \rightarrow \Seq(\bd) \times \Seq(\bd), \ x \mapsto \bigl(\SS(M_x), \SS^*(M_x)\bigr),$ was proved in \cite[Observation 2.10]{irredcompI}. 

For the sake of easy reference, we state a few basic facts regarding radical and socle layerings. 

\begin{Lemma}\label{lem-layering-prop}  Let $M, N \in \lamod$ with $\udim M = \udim N$.
\begin{itemize}
\item {\bf Duality:}  The radical and socle layerings are dual to each other, in the sense that
$$\SS(D(M)) = \bigl(D(\SS^*_0(M)), \cdots, D(\SS^*_L(M) \bigr)  \  \text{and}\, \  \SS^*(D(M)) =  \bigl(D(\SS_0(M)), \cdots, D(\SS_L(M) \bigr),$$ 
where $D$ denotes the duality $\Hom_K (- , K): \lamod \rightarrow \text{\rm{mod-}}\la.$

\item {\bf Radical layering:}  $\udim J^lM = \udim \bigoplus_{l \le j \le L}  \SS_j(M)$;
 in particular 
 $$\SS(M) \le \SS(N) \ \ \ \iff\ \ \ \udim J^l M \ge \udim J^l N \ \ \text{for all}\ \ l \in \{0, \dots, L\}.$$
 
\item {\bf Socle layering:}  $\soc_l M = \ann_M J^{l+1}$ and $\udim \soc_l M = \udim \bigoplus_{0\le j\le l}  \SS^*_j(M)$; in particular,
 $$\SS^*(M) \le \SS^*(N)\ \ \ \iff \ \ \  \udim \soc_l M \le \udim \soc_l N \ \ \text{for all}\ \  l \in \{0, \dots, L\}.$$
 
\item {\bf Connection:} $J^l M \subseteq \soc_{L - l} M$, and hence $\bigoplus_{l \le j \le L} \SS_j(M) \subseteq \bigoplus_{0 \le j \le L- l} \SS_j^*(M)$.
\end{itemize}
\end{Lemma}

 A semisimple sequence $\SS$ is called \textit{realizable} if there exists a left $\la$-module $M$ with $\SS(M) = \SS$.  The following criterion was proved in \cite[Criterion 3.2]{irredcompI}.  Here $\bB$ denotes the adjacency matrix of $Q$, i.e.,  $\bB_{ij}$ is the number of arrows from $e_i$ to $e_j$, and $P_1(\SS_l)$ is the first radical layer of a projective cover of $\SS_l$.

\begin{Lemma}\label{thm-realizability-criterion}
{\rm\textbf{Realizability Criterion}}. For a semisimple sequence $\SS = (\SS_0, \SS_1, \dots, \SS_L)$ over  a truncated path algebra, the following conditions are equivalent:  
\begin{itemize}
\item $\SS$ is realizable.
\item For each $l \in \{0, \dots, L - 1\}$, the two-term sequence $(\SS_l, \SS_{l+1})$ is realizable.
\item $\udim \SS_{l+1} \le \udim P_1(\SS_l)$ for $l \in \{1, \dots, L\}$, i.e., $\udim \SS_{l+1}\le \udim \SS_l \cdot \mathbf{B}$.
\end{itemize}
\end{Lemma}

An element $x$ of a $\la$-module $M$ is said to be \textit{normed} if $x = e_i x$ for some $i$.  A \textit{top element} of $M$ is a normed element  in $M \setminus JM$, and a \textit{full sequence of top elements of $M$} is any generating set of $M$ consisting of top elements which are $K$-linearly independent modulo $JM$. 

Given any subset $\U$ of $\modlad$, the modules corresponding to the points in $\U$ are called the modules ``in" $\U$.  When $\U$ is irreducible, the modules in $\U$ are said to \textit{generically have property} $(*)$ in case all modules in some dense open subset of $\U$ satisfy $(*)$.  Radical layerings and socle layerings, for instance, are generically constant on any irreducible subvariety of $\modlad$.  Hence it is meaningful to speak of \textit{the} generic radical and socle layerings of the irreducible components of $\modlad$.

\textit{Acknowledgments.} We wish to thank Eric Babson for numerous stimulating conversations on the subject of components at MSRI. Moreover, we thank the referee for his/her meticulous reading of the manuscript which led to significant improvements.
 The first author was partially supported by an NSF grant while carrying out this work.  While in residence at MSRI, Berkeley, both authors were supported by NSF grant 0932078 000.  The second author was also partially supported by NSF award DMS-1204733.

\section{Skeleta and generic modules over arbitrary basic algebras $\la$}\label{s-foundational}

The concepts of the title will play a key role in subsequent arguments.  Subsections B and C are new.

\subsection{The basics in compressed form}\label{ss-skeleta-gen-modules}

We recall the following definitions from \cite{BHT} and \cite{hier}.  Let  $\SS = (\SS_0, \dots, \SS_L)$ be a semisimple sequence, and $P$ a projective cover of $\SS_0$; in particular, $\udim P/JP = \udim \SS_0$.  Fix a full sequence $z_1, \dots, z_t$ of top elements of $P$, where $t = \dim \SS_0$; say $z_r = e(r) z_r$ with $e(r) \in \{e_1, \dots, e_n\}$, i.e., $P = \bigoplus_{1 \le i \le t} \la z_r$ with $\la z_r \cong \la e(r)$.   Skeleta are subsets of the projective $KQ$-module $\Phat = \bigoplus_{1 \le r \le t} KQ\, z_r$.  A \textit{path of length} $l$ in $\Phat$ is any element of the form $p z_r$, where $p$ is a path of length $l$ in $KQ \setminus I$ with $\start(p) = e(r)$.  In particular, the canonical image in $P$ of any path in $\Phat$ is nonzero. We do not make any formal distinction between the $\la$- and the induced $KQ$-module structure of a left $\la$-module, but rely on the context.  

\begin{enumerate}
\item[\bf1.] An (abstract) \textit{skeleton with layering $\SS$} is a set $\S$ consisting of paths in $\Phat$ which satisfies the following two conditions:  
\begin{itemize}
\item It is closed under initial subpaths, i.e., whenever $p z_r \in \S$, and $q$ is an initial subpath of $p$ (meaning $p = q' q$ for some path $q'$), the path $q z_r$ again belongs to $\S$. 

\item For each $l \in \{0, \dots, L\}$, the number of those paths of length $l$ in $\S$ which end in a particular vertex $e_i$ coincides with the multiplicity of $S_i$ in the semisimple module $\SS_l$.
\end{itemize}
In particular, a skeleton $\S$ with layering $\SS$ includes the paths $z_1, \dots, z_t$  of length zero in $\Phat$.  Note that the set of abstract skeleta with any fixed layering $\SS$ is finite.  

\item[\bf 2.] Suppose $M \in \lamod$.  We  call an abstract skeleton $\S$ a \textit{skeleton of} $M$ in case $M$ has a full sequence $m_1, \dots, m_t$ of top elements  with $m_r = e(r) m_r$ such that  
\begin{itemize}
\item $\{ p m_r \mid p z_r \in \S\}$ is a $K$-basis for $M$, and 
\item the layering of $\S$ coincides with the radical layering $\SS(M)$ of $M$.  
\end{itemize}
\end{enumerate}

Observe: For every $M \in \lamod$, the set of skeleta of $M$ is non-empty.  On the other hand,  the set of all skeleta of modules with a fixed dimension vector is finite.

The following is an excerpt of a result proved in \cite[Theorem 4.3]{BHT}; it is adapted to our present needs.  Let $K_0$ be the smallest subfield of $K$ such that $\la = KQ/I$ is defined over $K_0$; the latter means that $K_0 Q$ contains generators for $I$.  Moreover, let $\overline{K_0}$ be the algebraic closure of $K_0$ in $K$.  Clearly, $\overline{K_0}$ then has finite transcendence degree over the prime field of $K$.  Any automorphism $\phi \in \Gal(K{:}\overline{K_0})$ induces a ring automorphism $KQ \rightarrow KQ,\  \sum_i k_i p_i \mapsto \sum_i \phi(k_i) p_i$, which maps $I$ to $I$ and thus lifts to a ring automorphism of $\la$; the latter, in turn, gives rise to a Morita self-equivalence of $\lamod$, sending a module $M$ to the module whose $\la$-structure is that of $M$ twisted by $\phi$.  We refer to such a Morita equivalence as \textit{induced by} $\Gal(K{:}\overline{K_0})$.  Further, we call an attribute of a module \textit{$\Gal(K{:}\overline{K_0})$-stable} if it is preserved by all $\Gal(K{:}\overline{K_0})$-induced self-equivalences of $\lamod$.  Note that dimension vectors are $\Gal(K{:}\overline{K_0})$-stable, for instance; obviously, so are all properties that are invariant under arbitrary Morita equivalences. 
 
\begin{Theorem-Definition}\label{thm-gen-skeleta}
{\rm\textbf{Generic skeleta, existence and uniqueness of generic modules.}}
Assume that the field $K$ has infinite transcendence degree over its prime field.

Whenever $\C$ is an irreducible component of $\modlad$ with generic radical layering $\SS$ and $\S$ is a skeleton of \underbar{some} module in $\C \cap \laySS$, all modules in a dense open subset of $\C$ have skeleton $\S$.  In particular, it makes sense to speak of the \underbar{generic} \underbar{set} of skeleta of the modules in $\C$; it is the union of the sets of skeleta of the modules in $\C \cap \laySS$ and is $\Gal(K{:}\overline{K_0})$-stable.  

There exists a \underbar {generic} $\la$-\underbar{module} $G$ for $\C$, meaning that 
\begin{itemize}
\item $G$ belongs to $\C$ and 
\item $G$ has all $\Gal(K{:}\overline{K_0})$-stable generic properties of the modules in $\C$.  
\end{itemize}

 Generic modules are unique in the following sense:  Whenever $G$  and $G'$ are generic  for $\C$, there is a $\Gal(K{:}\overline{K_0})$-induced Morita equivalence $\lamod \rightarrow \lamod$ which takes the isomorphism class of $G$ to that of $G'$. \qed
\end{Theorem-Definition} 

For concrete illustrations of generic modules and generic skeleta see Section \ref{ss-example-hered} and Example \ref{ex-nonhered-1}.

Clearly, all Morita-invariant generic properties of the modules in $\C$ can be traced in any generic module $G$.  Beyond those:  Given a decomposition of $G$ into indecomposable direct summands, the collection of dimension vectors of the summands of $G$ is generic for $\C$ (see also \cite{KacII} and \cite{CBS}).  The same is true for the dimension vectors of the radical and socle layers of $G$.

\subsection{A crucial observation}\label{basefield}

In tackling the component problem, the following comments will allow us to assume without loss of generality that our base field $K$ has infinite transcendence degree over its prime field.  We will make this assumption whenever it is convenient to have generic objects $G \in \lamod$ for the components at our disposal. 

\begin{Observation}\label{com-gen-modules}\mbox{}
\begin{enumerate}

\item \textit{Passage to a base field of infinite transcendence degree over its prime field.}  Let  $\Khat$ be the algebraic closure of a purely transcendental extension field $K(X_\alpha {\, \mid \,} \alpha \in A)$ of $K$.  Then $\lahat : = \Khat \otimes_K \la$ is an algebra which has the same quiver (and hence the same dimension vectors) as $\la$; indeed, $\lahat  \cong \Khat Q / \widehat{I}$, where $\widehat{I}$ is the ideal of $\Khat Q$ generated by $I$; in particular $\lahat$ is truncated whenever $\la$ is.  The irreducible components of $\modlad$ are in natural inclusion-preserving one-to-one correspondence with those of $\rep_{\bd}(\lahat)$.  To see this, let $\Gamma$, resp.~$\widehat{\Gamma}$, be the coordinate ring of $\modlad$, resp.~of $\rep_{\bd}(\lahat)$.  The map $\spec \Gamma \rightarrow \spec \widehat{\Gamma}$, \, $\mathcal{P} \mapsto\widehat{\mathcal{P}} = \Khat \otimes_K \mathcal{P}\, $ is a well-defined inclusion-preserving injection; indeed, the tensor product $R_1 \otimes_K R_2$ of any two zerodivisor-free commutative algebras $R_1$, $R_2$ over an algebraically closed field $K$ is in turn a domain (see, e.g., \cite[Ch.III, Corollary 1 to Theorem 40]{ZaSa}).  This map restricts to a bijection on the set of minimal primes:  Namely, if $\mathcal{P}_1, \dots, \mathcal{P}_m$ are the minimal primes in $\spec \Gamma$,
then $(\widehat{\mathcal{P}}_1)^{r_1} \cdots (\widehat{\mathcal{P}}_m)^{r_m}  = 0$ for suitable $r_i \ge 0$, whence every minimal prime $\mathcal{Q} \in \spec \widehat{\Gamma}$ contains one of the $\widehat{\mathcal{P}}_i$; equality $\mathcal{Q} = \widehat{\mathcal{P}}_i$ follows.  

\item Now suppose that $\C$ is an irreducible component of $\modlad$, and let $\Ghat \in \lahat$-mod be a generic module for the corresponding irreducible component $\widehat{\C}$ of $\rep_{\bd}(\lahat)$.  Generically, the modules in $\C$ then have all those properties of $\Ghat$ which are reflected by the exact and faithful functor
$$\Khat \otimes_K - : \lamod \rightarrow \lahat\text{-mod}.$$
In particular, this pertains to dimension vectors, as well as skeleta and direct sum decompositions.  Crucial in the present context:  The dimension vectors of the generic radical and socle layers of $\C$ and $\widehat{\C}$ coincide.
\end{enumerate}
\end{Observation}

\subsection{Generic modules under duality}

Again, we let $\la = KQ/ I$ be an arbitrary basic algebra.  Since, in this subsection, we simultaneously consider left and right $\la$-modules, we will write $\Mod_\bd (\lamod)$ for $\modlad$, and $\Mod_\bd(\mbox{\rm mod-}\Lambda)$ for the analogous variety parametrizing the right $\la$-modules with dimension vector $\bd$ to emphasize sides.  The duality $D = \Hom_K( -, K):  \lamod \leftrightarrow \mbox{\rm mod-}\Lambda$ clearly gives rise to an isomorphism 
$$\Dhat: \Mod_\bd (\lamod) \longrightarrow  \Mod_\bd(\mbox{\rm mod-}\Lambda), \ \  (x_\alpha)_{\alpha \in Q_1} \mapsto (x_\alpha^t)_{\alpha \in Q_1},$$
where $x_\alpha^t$ is the transpose of $x_\alpha$.  In particular, $\Dhat$ induces a bijection between the irreducible components of the two varieties.  Moreover, one  observes that every automorphism $\phi \in \Gal(K{:}K_0)$ induces Morita self-equivalences $\Phi_{\rm{left}}$ of $\lamod$ and $\Phi_{\rm{right}}$ of $\text{mod-}\la$ such that $D \circ \Phi_{\rm{left}} =  \Phi_{\rm{right}} \circ D$.  Verification of the following fact is routine. 

\begin{Observation} \label{lem-dual} 
Let $\C$ be an irreducible component of $\Mod_\bd (\lamod)$.  Then a left $\la$-module $G$ is generic for $\Mod_\bd (\lamod)$ if and only if $D(G)$ is generic for $\Dhat(\C)$. \qed
\end{Observation}

\section{Generic modules over truncated path algebras}\label{s-truncated-path-algebras}

\textit{Throughout this section, $\la$ stands for a truncated path algebra based on an arbitrary quiver $Q$.}

\subsection{Projective presentations, submodules and quotients of generic modules}\label{subquots}

Theorem \ref{thm-gen-skeleta} has a useful supplement in the present setting.  Namely, it  permits us to pin down explicit generic minimal projective presentations of the modules in $\laySS$.  

First, we observe that the smallest subfield $K_0$ of $K$ over which $\la$ is defined is the prime field of $K$.  Moreover, involvement of the projective $KQ$-module $\Phat$ in the definition of a skeleton becomes superfluous; in fact, $\Phat$ may be replaced by the projective $\la$-module $P$.  The bonus of the truncated setting responsible for this simplification is the path length grading of $\la$; it leads to an unambiguous notion of \textit{length} of nonzero ``paths" of the form $p z_r \in P = \bigoplus_{1 \le r \le t} \la z_r$, where $p$ is a path in $Q$.  

For convenience, we will assume that the field $K$ has infinite transcendence degree over $K_0$; this guarantees that we can locate generic modules for the varieties $\laySS$ within the category $\lamod$.  In light of Observation \ref{com-gen-modules}, this assumption will not limit the applicability of our conclusions towards identifying the irreducible components of the varieties $\modlad$; nor will any of the considered generic properties of the modules in the components be affected by it.  

As in Subsection \ref{ss-skeleta-gen-modules}, we fix a realizable semisimple sequence $\SS$ and a distinguished projective cover $P = \bigoplus_{1 \le r \le t} \la z_r$ of $\SS_0$.  As explained above,  we may assume skeleta with layering $\SS$ to live in $P$.  Given a skeleton $\S$ with layering $\SS$, a path $\bq = q z_s \in P$ is called \textit{$\S$-critical} if it fails to belong to $\S$, while every proper initial subpath $\bq' = q' z_s$ belongs to $\S$.  Moreover, for any $\S$-critical path $\bq$, let $\S(\bq)$ be the collection of all those paths $\bp = p z_r$ in $\S$ which are at least as long as $\bq$ and terminate in the same vertex as $\bq$.   

\begin{Theorem} \label{thm-generic-modules}
{\rm\textbf{Generic modules for $\laySS$}. \cite[Theorem 5.12]{BHT}} Let $\la$ be a truncated path algebra, and suppose that the base field $K$ has infinite transcendence degree over its prime field $K_0$.  Moreover, let $\SS$ be a realizable semisimple sequence.  

If $\S$ is {\underbar{any}} skeleton with layering $\SS$, then the modules in $\laySS$ generically have skeleton $\S$, and the generic modules for $\laySS$ are {\rm (}up to $\Gal(K{:}\overline{K_0})$-induced self-equivalences of $\lamod${\rm{)}} determined by minimal projective presentations of the following format:
$P/ R(\sigma)$, where $P = \bigoplus_{1 \le r \le t} \la z_r$ is the distinguished projective cover of $\SS_0$, and  
$$R(\sigma) =   \sum_{\bq\  \S\text{-critical}} \la \  \biggl ( \bq  - \sum_{\bp \in \S(\bq)}  x_{\bq,\, \bp}  \ \bp\biggr)$$
for some family $\bigl (x_{\bq,\,  \bp} \bigr)$ of scalars algebraically independent over $K_0$.

Replacing $\bigl (x_{\bq,\, \bp} \bigr)$ in this presentation by an arbitrary family of scalars in $K$ results in a module in $\laySS$ with skeleton $\S$, and conversely, every module with skeleton $\S$ is obtained in this way.   \qed 
\end{Theorem}

The following consequences of Theorem \ref{thm-generic-modules} are new.  Apart from being of interest in their own right, they set the stage for inductive arguments.  
Given any module $N \in \lamod$, we call a submodule $M$ \textit{layer-stably embedded} in $N$ in case $M \cap J^l N = J^l M$ for all $l \ge 0$; this means that, canonically, $\SS_l(M) \subseteq \SS_l (N)$.  

\begin{Corollary}\label{cor-subs-quots-gen-mods}
{\rm\textbf{Submodules and quotients of generic modules}}.   Let $\la$ be a truncated path algebra, $\SS = (\SS_0, \dots, \SS_L)$ a realizable semisimple sequence, and $G \in \lamod$ a generic module for $\laySS$.   
\begin{enumerate}
\item[\rm(a)] If $U$ is a submodule of $G$ which is layer-stably embedded in some $J^m G$, then $U$ is a generic module for $\rep\, \SS(U)$.  In particular:  $J^m G$ is a generic module for $\rep\, (\SS_m, \dots, \SS_L, 0, \dots, 0)$ whenever $0 \le  m \le L$.
\item[\rm(b)] Whenever $0 < m \le L$, the quotient $G / J^m G$ of $G$ is a generic module for $$\rep\, \bigl(\SS_0, \dots, \SS_{m-1}, 0, \dots, 0\bigr).$$  
\end{enumerate}
\end{Corollary}

\begin{proof}  (a) In a preliminary step, we verify the special case where $U = J^m G$.    In this case, $\SS(U) = \bigl(\SS_m, \dots, \SS_L, 0, \dots, 0 \bigr)$.  We set $\SS' = \SS(U)$.

Fix a projective cover $P = \bigoplus_{1 \le r \le t} \la z_r$ of $\SS_0$ and a projective cover $P' = \bigoplus_{1 \le r \le u} \la z'_r$  of $\SS_m$.  Moreover, let $\S = \bigsqcup_{0 \le l \le L} \S_l \subseteq P$ be a skeleton of $G$ such that $G$ has a presentation as specified in Theorem \ref{thm-generic-modules} relative to $\S$; here $\S_l$ denotes the set of all paths of length $l$ in $\S$.  Set $\S' = \bigsqcup_{m \le l \le L} \S_l$, and identify the paths in $\S_m$ with the distinguished top elements $z'_1, \dots, z'_u$ of $P'$.  Under this identification, we find $\S'$ to be a skeleton of $J^m G$.  Since $\la$ is a truncated path algebra, the $\S'$-critical paths in $P'$ are then in an obvious one-to-one correspondence with those $\S$-critical paths that have length $\ge m+1$:  Indeed, given any $\S$-critical path $q z_s$ of length $> m$, replace its initial subpath of length $m$ by the appropriate $z'_j$ (any such initial subpath belongs to $\S_m$ by the definition of criticality) to arrive at a $\S'$-critical path $q' z'_j$; it is routine to check that this yields a bijection as stated.  Hence the description of $G$ in Theorem \ref{thm-generic-modules} shows $J^m G$ to be generic for $\SS'$; the role played by $\S$ in the considered presentation of $G$ is taken over by the skeleton $\S'$ with layering $\SS'$.  This proves our claim in case $U = J^mG$.

To complete the proof of (a), it thus suffices to show the following:  Any layer-stably embedded submodule $U$ of $G$ is generic for $\SS(U)$.  Again, let $z_1, \dots, z_t$ be the fixed full sequence of top elements of the projective cover $P$ of $\SS_0$ which provides the coordinate system for skeleta with layering $\SS$.  We may assume that $\dim U/JU = u \ge 1$, and that the distinguished projective cover $Q$ of $\SS_0(U)$, on which we base the skeleta with layering $\SS(U)$,  is of the form $Q = \bigoplus_{1 \le r \le u} \la z_r$;  this assumption is justified by the inclusion $\SS_0(U) \subseteq \SS_0$.  Pick any skeleton $\S(U) \subseteq Q$ of $U$; say $\S(U) = \bigsqcup_{0 \le l \le L} \S_l(U)$, where $\S_l(U)$ is the set of paths of length $l$ in $\S(U)$.   We embed $\S(U)$ into a skeleton $\S$ of $G$ as follows: In light of $\SS_1(U) \subseteq \SS_1$, we may supplement $\S_1(U)$ to a $K$-basis consisting of paths of length $1$ in $P$.  Moreover, since $U$ is a submodule of $G$ such that $JU/J^2 U$ canonically embeds into $JG/ J^2 G = \SS_1$, we may arrange for the paths $\S_1 \setminus \S_1(U)$ to all start in one of the top elements $z_{u+1}, \dots, z_t$ of $P$.  Invoking the facts that $\SS_2(U) \subseteq \SS_2$ and $U$ is closed under multiplication by paths, we may supplement $\S_2(U)$ to a basis $\S_2$ for $\SS_2$ such that each path  in $\S_2 \setminus \S_2(U)$ extends one of the paths in $\S_1 \setminus \S_1(U)$; in particular each path in $\S_2 \setminus \S_2(U)$ starts in one of the vertices $z_{u+1}, \dots, z_t$.  Proceeding recursively, we thus arrive at a skeleton $\S$ of $G$ such that $\S \setminus \S(U)$ consists of paths in $\bigoplus_{u+1 \le r \le t} \la z_r$.   In this situation, the $\S(U)$-critical paths are precisely those $\S$-critical paths in $P$ which start in one of the top elements $z_1, \dots, z_u$.  

By the uniqueness statement of Theorem \ref{thm-gen-skeleta}, $G$  has a projective presentation $P/R(\sigma)$, as described in Theorem \ref{thm-generic-modules}, based on the skeleton $\S$ we just constructed.  Since the residue classes $\overline{pz_r}$ of the $pz_r \in \S$ form a basis for $G$, and the classes represented by the $pz_r$ in $\S(U)$ generate $U$, we conclude that, for any $\S(U)$-critical path $q z_s$, the set $\S(qz_s)$ is contained in $\S(U)$.  Thus Theorem \ref{thm-generic-modules} exhibits $U$ as generic for $\rep\, \SS(U)$ in the present situation.  

Part (b) is proved analogously. 
\end{proof}

Part (b) of Corollary \ref{cor-subs-quots-gen-mods} cannot be upgraded to a level matching part (a):
If $G$ is as in the corollary and $U \subseteq J^m G$ is layer-stably embedded in $J^m G$, then $G/U$ need not be generic for the radical layering of $G/U$.   For instance:

\begin{Example}\label{ex-quot-gen-not-gen}  Let $\la = KQ$ be the Kronecker algebra, i.e., $Q$ is the quiver with two vertices, $e_1$ and $e_2$ say, and two arrows $\alpha_1, \alpha_2$ from $e_1$ to $e_2$.  Then $G = \la e_1$ is generic for $\SS = (S_1, S_2^2)$, and $U = \la \alpha_2$ is layer-stably embedded in $JG$.  However, $G/U$ fails to be generic for $\SS' = \SS(G/U) = (S_1, S_2)$; indeed, whenever $G'$ is generic for $\SS'$, we have $\alpha_2 G' \ne 0$.  
\end{Example}

On the other hand, Corollary \ref{cor-subs-quots-gen-mods}(a) yields the generic property for further quotients of a generic module $G$ by way of the duality of Section 2.D.  

\begin{Corollary} {\rm{\bf Duality and socle quotients.}} \label{socquots}
\begin{enumerate}
\item[\rm(a)] Let $\SS= (\SS_0, \dots, \SS_L)$ be any semisimple sequence in $\lamod$, $G$ a generic module for $\laySS$, and $\SS^* = \SS^*(G)$.  Then the generic socle layering {\rm {(}}resp.~the generic radical layering{\rm {)}} of the modules in $\Dhat(\laySS)$ is the semisimple sequence $D(\SS) = \bigl(D(\SS_0), \dots, D(\SS_L) \bigr)$ {\rm {(}}resp.~$D(\SS^*) = \bigl(D(\SS^*_0), \dots, D(\SS^*_L) \bigr)${\rm {)}}.  Moreover, $G/\soc G$ is a generic module for $\Mod \SS'$, where $\SS' = \SS(G/\soc G)$.    
\item[\rm(b)] Let $\C = \overline{\laySS}$ be an irreducible component of $\modlad=  \Mod_\bd(\lamod)$ and $\SS^*$ the generic socle layering of the modules in $\C$.   Then the irreducible component $\Dhat(\C)$ of $\Mod_\bd(\mbox{\rm mod-}\Lambda)$ is the closure in $\Mod_\bd(\mbox{\rm mod-}\Lambda)$ of the subvariety consisting of the modules with radical layering $D(\SS^*)$.   
\end{enumerate}
\end{Corollary}

\begin{proof}  (a) In light of Lemmas \ref{lem-layering-prop} and \ref{lem-dual}, we obtain: $G/ \soc G$ is generic for $\Mod \SS'$ if and only if $D(G/\soc G) \cong D(G) J$ is generic for $\Mod \bigl(\SS(D(G) J )\bigr)$. Since the generic property of the module $D(G) J$ for its radical layering 
was shown in Corollary \ref{cor-subs-quots-gen-mods}(a), the first claim follows. 
Part (b) is immediate from the cited lemmas.
\end{proof}

The generic projective presentations of the modules in $\laySS$ exhibited in Theorem \ref{thm-generic-modules} permit us, moreover, to compute the generic format of endomorphisms of the modules in $\laySS$.  In particular, this yields the generic dimension of endomorphism rings.
  
\begin{Corollary} {\rm{\bf Generic endomorphism rings.}}\label{cor-gen-endo}  Let $\SS$ be a realizable semisimple sequence with dimension vector $\bd$ over a truncated path algebra $\la$.  The generic dimension of $\End_\la(M)$ for $M$ in $\overline{\laySS}$ may be determined from $\SS$ and $Q$ by way of a system of homogeneous linear equations.  
\end{Corollary}

\begin{proof}  We refer to the notation of Theorem \ref{thm-generic-modules}.  In particular, we let $\S$ be a skeleton with layering $\SS$, and use the generic form of the minimal projective presentations provided by the theorem.  Clearly any endomorphism $\phi$ of $P/R(\S)$ is completely determined by the (unique) scalars arising in the equations $\phi(z_j) = \sum_{u \in \S} k_{\, j,\, u}\, u$ $(\dagger)$.  Suppose moreover that, for any path $\ptilde \in KQ \setminus I$ and $u \in \S$, the product $\ptilde\, u$ expands in the form $\ptilde\, u = \sum_{v \in \S} c(\ptilde, u, v)\, v$ $(\ddagger)$ with $c(\ptilde, u, v) \in K$.  That $\phi$ be an endomorphism of $P/R(\S)$ is equivalent to the following equalities:
$$q\, \phi( z_s) \ \ = \ \   \sum_{p z_r \in \S(q z_s)}  x_{q z_s,\, p z_r}  \ p\, \phi( z_r) \ \ \ \ \ \text{for all $\S$-critical paths}\  q z_s.$$
Expanding both sides of these equalities by first inserting $(\dagger)$, then following with $(\ddagger)$, one obtains $K$-linear combinations of the paths in $\S$ on either side.  Comparing coefficients of these basis expansions results in a system of linear equations for the decisive scalars $k_{\, j,\, u}$. 
\end{proof}  

\subsection{Generic socle series for $\mathbf{Rep}\, \mathbb{S}$}\label{ss-gen-soc-rep-S} 

In order to effectively apply the Main Theorem, one needs to determine the generic socle series $\SS^*$ for each of the varieties $\laySS$.  Indeed, upper semicontinuity of the map $\modlad \rightarrow \Seq(\bd)$, $x \mapsto \SS^*(M_x)$, implies that $\SS^*$ is the unique smallest semisimple sequence with $(\SS, \SS^*) \in \radsocbd$.

For any $M \in \lamod$, we write $E_1(M) = \soc (E(M)/M)$, where $E(M)$ is an injective envelope of $M$.  If $M$ is semisimple, this socle is transparently encoded in the quiver $Q$.  Indeed, let $\mathbf{A}$ be the transpose of the adjacency matrix of $Q$, i.e., $\mathbf{A}_{ij}$ is the number of arrows from $e_j$ to $e_i$.
Then, given any semisimple module $T$, the corresponding semisimple $E_1(T)$ is determined by its dimension vector
$$\udim E_1(T)\  =\  \udim T \cdot \mathbf{A}.$$

As we identify isomorphic semisimple modules, the set of semisimples in $\lamod$ is a lattice under the componentwise partial order of dimension vectors.  The join $\sup(T, T')$ of $T$ and $T'$ is pinned down by its dimension vector $\sup(\udim T, \udim T')$; analogously for the meet $\inf(T,T')$. 
\smallskip

{\it For the remainder of Section 3, we fix a realizable semisimple sequence $\SS = (\SS_0, \dots, \SS_L)$ and a generic $\la$-module $G$ for $\laySS$.\/}  By Observation \ref{com-gen-modules}, we do not lose any generality in assuming existence of $G$, since the dimension vectors of the socles we wish to determine are not affected by passage to a potentially enlarged base field.

\begin{Lemma}  \label{lemprep1}
   Let $L(G)$ be the Loewy length of $G$, and $l \in \{0, \dots, L(G) - 2 \}$. If $U$ is a submodule of $J^lG$ such that  $J^{l+1} G \subseteq U$ is an essential extension, then the following conditions are equivalent:
\smallskip

{\rm (1)}  $U$ is a maximal essential extension of  $J^{l+1} G$ in $J^lG$.

{\rm (2)}  $J^l G = U \oplus C_l$ for some semisimple module $C_l \subseteq \SS_l$.

{\rm (3)}  $JU = J^{l+1} G$, and $U/JU = \inf\, \bigl( \SS_l,\, E_1(J^{l+1} G) \bigr)$.
\smallskip

The {\rm{(}}isomorphism classes of the{\rm{)}} semisimples $C_l$ are independent of the choices of $U$ satisfying {\rm (2)} and the blanket hypothesis.  In fact, the $C_l$ are the maximal semisimple direct summands of the radical powers $J^l G$, respectively.  In particular, $\soc (J^{L -m} G) = \bigoplus_{0 \le j \le m} C_{L - j}$, where $C_l = J^l G$ for $l > L(G) - 2$.
\end{Lemma}

\begin{proof} In light of Corollary \ref{cor-subs-quots-gen-mods}, it suffices to prove the lemma for $l = 0$.  The equivalence of (1) and (2) is straighforward for arbitrary $G$. 
\smallskip 

The first of the following observations hinges on the assumption that $G$ is generic for $\laySS$. 

$(\dagger)$ Whenever there exists {\it some\/} essential extension $JG \subseteq V$ such that $JV = JG$ and $V/ JG$ is semisimple with $\udim V/ JG \le \udim \SS_0$,  there exists a matching essential extension $JG \subseteq U$ inside $G$, namely an extension with the property that $\SS(U) = \SS(V)$.  This is an immediate consequence of Theorem \ref{thm-generic-modules}.

$(\ddagger)$ If $U$ is an essential extension of $JU = JG$ in $G$, we have $U/JU \le  \inf \bigl( \SS_0,\, E_1(J G)\bigr)$.  Indeed, we may identify $U$ with a submodule of $E(JG)$, which shows $U/JU$ to embed into $G/JG$, as well as into $E(J G) / JG$.  Due to semisimplicity of $U/JU$, this quotient actually embeds into $E_1(J G)$.

``(1)$\implies$(3)":  Assume that $U$ satisfies (1), and hence also (2), for $l=0$.  Write
$G = U \oplus C_0$ for some semisimple $C_0$.  In particular, $JU = JG$, whence $U/JU \le  \inf \bigl( \SS_0,\, E_1(J G) \bigr)$ by $(\ddagger)$.  Clearly, we may choose an extension $V_0$ of $JG$ in $E(JG)$ such that $V_0/JG$ is semisimple and equals $\inf \bigl( \SS_0,\, E_1(JG) \bigr)$.  In view of $(\dagger)$ and the fact that $U$ differs from $G$ by a semisimple direct summand only, we find an essential extension $U_0$ of $JG$ inside $U$ with the property that $\SS(U_0)  =   \SS(V_0)$; in particular $JU_0 = JG$.  Combining the equality $\SS(U_0)  =   \SS(V_0)$ with the first part of the argument, we thus obtain $ \udim V_0 / JG = \udim U_0 / JG \le \udim U /JG  \le \inf \bigl(\udim \SS_0,\, \udim E_1(J G)\bigr) = \udim V_0/JG$.

Verification of ``(3)$\implies$(1)" and of the supplementary claim is now routine. 
\end{proof}

In order to compute the dimension vectors of the semisimples $C_l$ of Lemma \ref{lemprep1}, we recursively define (isomorphism classes of) submodules $\SS'_m \subseteq \SS_m$, next to semisimple modules $D_m \subseteq \bigoplus_{L - m \le l \le L} \SS_l$, as follows:  
\begin{itemize} 
\item $\SS'_L = D_L = 0$.
\item $\SS'_{L-j}  = \inf\biggl(\SS_{L-j},\, \bigl(E_1(\SS_{L-j+1}) \oplus D_{L-j+1} \bigr)\biggr)$, and $D_{L-j}$ is defined by the requirement 
that \linebreak $\SS'_{L-j}  \oplus D_{L-j} = E_1(\SS_{L-j+1}) \oplus D_{L-j+1}$. 
\end{itemize}

\begin{Lemma} \label{lemprep2}
 Let  $\SS'_l$ be as above.   Then 
$$\udim \soc(J^{L-m} G) = \sum_{0 \le j \le m} \bigl(\udim \SS_{L-j} - \udim \SS'_{L-j} \bigr) \ \ \ \text{for} \ 0 \le m \le L.$$
The $C_l$ appearing in {\rm Lemma \ref{lemprep1}} have dimension vectors $\udim \SS_{l} - \udim \SS'_{l}$ and may thus be identified with direct complements of $\SS'_l$ in $\SS_l$, respectively.  
\end{Lemma}

\begin{proof}  Using Lemma \ref{lemprep1}, one  proves the claim by induction on $m \ge 0$, in tandem with the equalities $E_1(\SS_{L-m}) \oplus D_{L-m} = E_1(J^{L-m} G)$,\ $\udim D_{L-m} = \udim E_1(J^{L-m +1} G) - \udim \SS'_{L -m}$, and $\udim C_{L-m} = \udim \SS_{L-m} - \udim \SS'_{L-m}$.       
\end{proof}

The upcoming theorem doubles as an algorithm for determining the generic socle layering $\SS^*$ of the modules in $\laySS$ from the (realizable) semisimple sequence $\SS$. 
Recall that $\mathbf{A}$ is the transpose of the adjacency matrix of $Q$.

\begin{Theorem}  \label{mainsoc}
  We continue to assume that $\la$ is truncated and that $G$ is a generic module for $\laySS$.   
\smallskip

{ \bf (a)}  Set $\SS_{L+1} = 0$.  For $0 \le m \le L$,
$$\udim \soc (J^{L - m} G) \  =  \  \sup\, \biggl\{ \sum_{L-j \le l \le L} \biggl( \udim \SS_l - \udim \SS_{l+1} \cdot \mathbf{A} \biggr) \biggm| 0 \le j \le m \biggr\}.$$
The generic socle $\SS^*_0$ of the modules in $\laySS$ arises as the special case $m = L$. 
\smallskip
 
{\bf (b)} The higher entries of the generic socle layering $(\SS^*_0, \dots, \SS^*_L)$ for $\laySS$ are obtained via a recursion based on the following facts:  Generically, the quotients $M/ \soc M$ for $M$ in $\laySS$ have radical layering $\SS' = (\SS_l / C_l)_{0 \le l \le L-1}$ {\rm (cf.~Lemmas \ref{lemprep1} and \ref{lemprep2})}. Moreover, the generic socle layering of the modules in $\rep{\SS'}$ is $(\SS^*_1, \dots, \SS^*_L, 0)$.  
\end{Theorem}

\begin{Remark*} We will use the following abbreviation: 
$$\bdel_{L- j} \ \ = \ \ \sum_{L-j \le l \le L} \bigl( \udim \SS_l - \udim \SS_{l+1} \cdot \mathbf{A} \bigr) \ \  \in \  \ZZ^n.$$ 
Observe that these vectors may have negative entries in general.  However, the suprema appearing  in the theorem are nonnegative, because $\bdel_{L} = \udim \SS_L$ is among the contending vectors. 
\end{Remark*} 

\begin{proof}  (a) We adopt the notation of Lemmas \ref{lemprep1} and \ref{lemprep2}.  These lemmas tell us that 
\begin{equation}\soc (J^{L - m} G) = \bigoplus_{0 \le j \le m} C_{L - j} \ \ \  \text{with}\  \ \ \udim C_l = \udim \SS_l - \udim \SS'_l. \tag{1}\end{equation} 

For $k \in \{1, \dots, n\}$ and $\mathbf{b} \in \ZZ^n$, we write $[\mathbf{b}]_k$ for the $k$-th component of $\bb$; moreover, for $M \in \lamod$, we set $[M]_k: = [\udim M]_k = \dim e_k M$. 
\smallskip

{\bf I. Auxiliaries}.   Taking into account that $\SS_{L+1} = \SS'_L =  D_L =  0$, we allow for slight redundancies in our formulas to make them more symmetric.  For $1 \le j \le L$,
\begin{equation} \SS'_L \oplus \SS'_{L-1} \oplus \cdots \oplus \SS'_{L - j} \oplus D_{L-j} = E_1(\SS_{L+1} \oplus \SS_L \oplus \cdots \oplus \SS_{L - j + 1}).  \tag{2} \end{equation}
This equality is readily derived from the definitions by a subsidiary induction.
\smallskip

By adding $C_L \oplus \cdots \oplus C_{L-j}$ to both sides of $(2)$, we obtain:  
$$\SS_L \oplus \SS_{L-1} \oplus \cdots \oplus \SS_{L - j} \oplus D_{L-j} = C_L \oplus \cdots \oplus C_{L - j} \oplus E_1(\SS_{L+1} \oplus \SS_L \oplus \cdots \oplus \SS_{L - j + 1}),$$
which, in turn, implies
\begin{equation} [\bdel_{L- j}]_k = [C_L \oplus \cdots \oplus C_{L - j}]_k -  [D_{L-j}]_k. \tag{3} \end{equation}
for all $j \le L$.
\smallskip

{\bf II. The case of nonvanishing $[C_{l}]_k$}: \  If $[C_{L-v}]_k \ne 0$, then $[D_{L-v}]_k =  0$. 
\smallskip
Indeed, $[C_{L-v}]_k \ne 0$ implies that $[\SS'_{L-v}]_k  <  [\SS_{L-v}]_k$, whence the definition of $\SS'_{L-v}$ yields $[\SS'_{L-v}]_k = [E_1(\SS_{L -v+1}) \oplus D_{L - v+1}]_k$.  From $[\SS'_{L -v} \oplus D_{L-v}]_k  = [E_1(\SS_{L - v+1}) \oplus D_{L - v+1}]_k$ we thus infer $[D_{L-v}]_k = 0$. 

In view of (1), equality (3) thus reduces to
\begin{equation} [\bdel_{L-v}]_k = [C_{L} \oplus \cdots \oplus C_{L-v}]_k = [\soc J^{L-v} G]_k .\tag{4} \end{equation} 
\smallskip 
\smallskip

{\bf III. The principal induction}.  We prove  
$\udim \soc (J^{L - m} G) \  = \  \sup\{\bdel_{L - j} \mid 0 \le j \le m \}$
by induction on $m \ge 0$.  The case $m = 0$ being obvious, we assume the equality for some nonnegative $m < L$. Our claim amounts to 
$$[\soc J^{L - (m+1)} G]_k  = \max\, \{ \partial_{L - j} \mid 0 \le j \le m + 1\} \ \ \text{for} \ 1 \le k \le n.$$
We keep $k$ fixed in the following.  Due to (1), $[\soc J^{L - (m+1)} G]_k = \sum_{0 \le j \le m+1} [C_{L-j}]_k$.  
\smallskip

$\bullet$ First, we deal with the case where $[\soc J^{L - (m+1)} G]_k  = 0$, i.e. $[C_{L-j}]_k = 0$, for $0 \le j\le m+1$.   Invoking (3), we derive
$[\bdel_{L- j}]_k = - [D_{L-j}]_k \le 0 \ \ \text{for}\  j \le m+1$, and our claim follows.
\smallskip

$\bullet$ Next we address the case where $[\soc J^{L-(m+1)} G]_k \ne 0$, but $[C_{L-(m+1)}]_k  = 0$.  Then 
$$[\soc J^{L - (m+1)} G]_k = [\soc J^{L - m} G]_k$$
by (1).  In particular, $[\soc J^{L - m} G]_k  \ne 0$ in the present situation.  
Let $u$ be minimal with the property that $[\soc J^{L - (m+1)} G]_k = [\soc J^{L - u} G]_k$.  Then $0 \le u \le m$, and $[C_{L-u}]_k \ne 0$, while $[C_{L - (u+ 1)}]_k = \cdots = [C_{L - (m+1)}]_k = 0$.  From $[C_{L-u}]_k \ne 0$, we obtain  $[\soc J^{L-u} G]_k = [C_L \oplus \cdots \oplus C_{L-u}]_k = [\bdel_{L-u}]_k$ by Step II.  We combine the vanishing of the listed dimensions $[C_l]_k$ with (3) and (4) to infer 
$$[\bdel_{L- (m+1)}]_k = [ C_{L} \oplus  \cdots \oplus C_{L - u}]_k -  [D_{L-(m+1)}]_k =  [\bdel_{L- u}]_k - [D_{L-(m+1)}]_k \le [\bdel_{L- u}]_k.$$
Hence,
$$[\soc J^{L - (m+1)} G]_k = [\soc J^{L - m} G]_k = \max\{[\bdel_{L-j}]_k \mid 0 \le j \le m\} = \max\{[\bdel_{L-j}]_k \mid 0 \le j \le m+1\}$$
by the induction hypothesis.  

\smallskip
$\bullet$  Finally, we assume $[C_{L-(m+1)}]_k \ne 0$. This implies $[D_{L-(m+1)}]_k = 0$ by Step II, and consequently $[\soc J^{L - (m+1)} G]_k = [\bdel_{L-(m+1)}]_k$ by (4).  The induction hypothesis guarantees that
$\max\{[\bdel_{L-j}]_k \mid 1 \le j \le m\} = [\soc J^{L - m} G]_k <  [\soc J^{L - (m+1)} G]_k$, and therefore $\max\{[\bdel_{L-j}]_k \mid 1 \le j \le m+1\} = [\bdel_{L - (m+1)}]_k$.
This proves our claim.  
\medskip

(b)  In view of Lemma \ref{lemprep1}, $\soc J^{L-m} G$ is the internal direct sum $\bigoplus_{0 \le j \le m} C_{L-j}$ of submodules of $G$ for each $m \le L$ (not only up to isomorphism), and consequently $G/ \soc G$ indeed has radical layering $\SS'$.  That $G/ \soc G$ is generic for $\Mod \SS'$ thus follows from Corollary \ref{socquots}.  
\end{proof}

\begin{Corollary} \label{socetc}
 Retain the hypotheses and notation of {\rm Theorem \ref{mainsoc}}.  Then $C_m$ is the largest semisimple module which generically occurs as a direct summand of $J^m G$ under our partial order on semisimple modules.  In particular, $C_0$ is the largest semisimple module generically occurring as a direct summand of the modules in $\laySS$.  Each $C_m$ is determined by its dimension vector $\udim C_m = \udim \soc J^{L-m} G - \udim \soc J^{L-m + 1} G$.
\end{Corollary}

\section{Application to the hereditary case}\label{s-hereditary}

In this section, we indicate how the information for truncated path algebras which we
have assembled plays out in the special case where $\la$ is  hereditary.   Thus \textit{we assume that $\la = KQ$, where $Q$ is a quiver without oriented cycles.}  Moreover, we let $L$ be the maximum of the path lengths in $Q$; in particular, this entails $J^{L+1} =  0$.  Recall that the varieties $\modlad$ are full affine spaces in the present situation.  

 Once the generic radical layering $\SS$ of the $\bd$-dimensional representations of $Q$ is available, the results of Section \ref{s-truncated-path-algebras} will provide us with the set of generic skeleta  for $\modlad$, as well as with generic minimal projective presentations of the modules in $\modlad$.  This information, in turn, serves as a vehicle for accessing further generic data on the $\bd$-dimensional modules; some of them may alternatively be obtained by the methods of Kac and Schofield (\cite{KacII} and \cite{Scho}).  

\subsection{The generic radical layering $\SS$ of the modules in {\bf Rep}$_{\mathbf{d}}(\Lambda)$}
\label{ss-finding-gen-rad-lay}

First we show how to obtain the generic radical layering $\SS$ of the modules in $\modlad$ (i.e., the unique minimal radical layering of the $\bd$-dimensional $\la$-modules) directly from $\bd$,  without resorting to comparisons.  If $G$ is a generic module for $\modlad$, then the radical $J G$ is generic for its dimension vector by Corollary \ref{cor-subs-quots-gen-mods}.  Therefore, the gist of the task is to compute the generic tops of the modules with dimension vector $\bd$.

\begin{Proposition}\label{thm-generic-tops}
The generic top of a $\la$-module with dimension vector $\bd$ has dimension vector 
$$\mathbf{t} = \sup\, \{\mathbf{0}, \, \bd - \bd \cdot \mathbf{B} \},$$
where $\mathbf{B}$ is the adjacency matrix of $Q$, and the supremum refers to the componentwise partial order on $\ZZ^{|Q_0|}$.
\end{Proposition}

\begin{proof}
Let $Q_1$ be the set of arrows of $Q$, and endow the vector space $KQ_1$ on the basis $Q_1$ with the obvious $KQ_0$-bimodule structure.  Given a $KQ_0$-module $M$ with dimension vector $ \bd$, a $\la$-module structure on $ M $ is given by a $ K Q_0 $-linear map
\[ \alpha \colon K Q_1 \otimes_{K Q_0} M \to M; \]
the map 
$\alpha$ is not subject to any further constraints since $\la = KQ$.  Observe that $ JM $ and $ M/JM$ are the image and cokernel of $ \alpha $, respectively.  Now, $ \bd \cdot \mathbf{B} $ is the dimension vector of $ K Q_1 \otimes_{K Q_0} M $.  So the rank of $ \alpha $ is at most $ \inf\{\bd, \bd \cdot \mathbf{B}\} $, and this lower bound is attained for general $ \alpha $.  The claim is now immediate.
\end{proof}

In the above notation, the generic radical of the modules in $\modlad$ has dimension vector $\udim JG = \bd - \mathbf{t}$, and the first generic radical layer $\SS_1$ of these modules has dimension vector $\mathbf{t}^{(1)} = \udim JG/J^2 G$.  Given that $JG$ is generic for the modules with dimension vector $\bd - \mathbf{t}$ by Corollary \ref{cor-subs-quots-gen-mods}, the theorem therefore yields $\mathbf{t}^{(1)} = \sup\, \{\mathbf{0}, \, (\bd - \mathbf{t}) - ( \bd - \mathbf{t}) \cdot \mathbf{B} \}$, whence iterated application leads to the following recursion: 

\begin{Corollary}\label{cor-generic-tops}
The generic radical layering $\SS$ of the $\la$-modules with dimension vector $\bd$ is pinned down by the dimension vectors $\mathbf{t}^{(l)} = \udim \SS_l$, for $0 \le l \le L$, where 
$\mathbf{t}^{(0)}  = \sup\, \{\mathbf{0}, \, \bd - \bd \cdot \mathbf{B} \}$, and
$$\mathbf{t}^{(l+1)} = \sup\, \biggl\{\mathbf{0}, \, \biggl(\bd - \sum_{i \le l} \mathbf{t}^{(i)}\biggr) - \biggl(\bd - \sum_{i \le l} \mathbf{t}^{(i)} \biggr)  \cdot \mathbf{B} \biggr\}.$$
\end{Corollary}

\subsection{An example illustrating the theory in the hereditary case}\label{ss-example-hered}
Let $\la = \CC Q$, where $Q$ is the quiver
$$\xymatrixrowsep{0.5pc}\xymatrixcolsep{3.6pc}
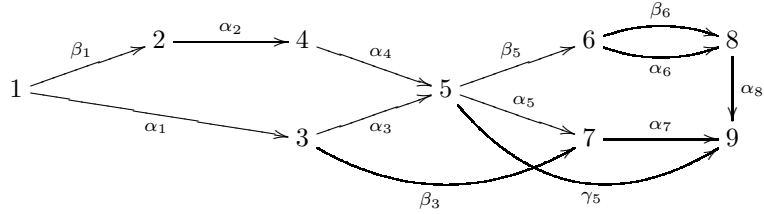
\xymatrix{
 &2 \ar[r]^{\alpha_2} &4 \ar[dr]^{\alpha_4} &&6 \ar@/^/[r]^{\beta_6} \ar@/_/[r]_{\alpha_6} &8 \ar[dd]^{\alpha_8}  \\
1 \ar[ur]^{\beta_1} \ar[drr]_{\alpha_1} &&&5 \ar[ur]^{\beta_5} \ar[dr]^{\alpha_5} \ar@/_6ex/[drr]_(0.55){\gamma_5}  \\
 &&3 \ar[ur]_{\alpha_3} \ar@/_4ex/[rr]_(0.45){\beta_3} &&7 \ar[r]^{\alpha_7} &9
 }$$
and let $\bd = (0,1,1,0,3,2,3,5,10) \in (\NN_0)^9$.  Proposition \ref{thm-generic-tops} implies that $\SS_0 = S_2 \oplus S_3 \oplus S_5^2 \oplus S_8$ is the generic top of the $\la$-modules with dimension vector $\bd$.  Going over the same sequence of steps with $\bd^{(1)} =  (0,0,0,0,1, 2,3,4,10)$, we obtain $\SS_1$, and so forth.  The resulting generic sequence $\SS$ of radical layers may be read off any of the generic skeleta of the modules in $\modlad$. We present one of them for further discussion.
$$\xymatrixrowsep{0.75pc}\xymatrixcolsep{1pc}
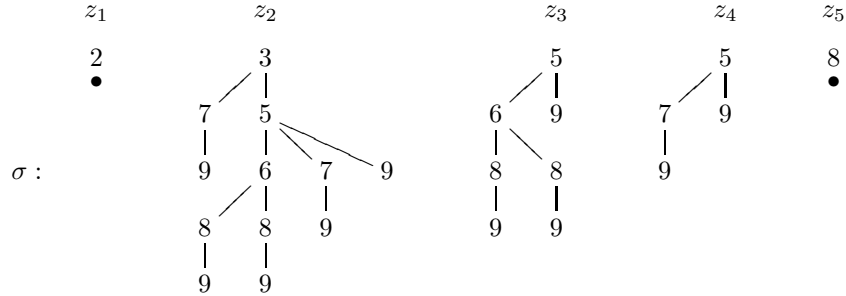
\xymatrix{
 &2 \dropup{z_1} \drbl &&&3 \dropup{z_2} \edge[dl] \edge[d] &&&&&5 \dropup{z_3} \edge[dl] \edge[d] &&&5 \dropup{z_4} \edge[dl] \edge[d] &&8 \dropup{z_5} \drbl  \\
&&&7 \edge[d] &5 \edge[d] \edge[dr] \edge[drr] &&&&6 \edge[d] \edge[dr] &9 &&7 \edge[d] &9  \\
\sigma:  &&&9 &6 \edge[dl] \edge[d] &7 \edge[d] &9 &&8 \edge[d] &8 \edge[d] &&9  \\
 &&&8 \edge[d] &8 \edge[d] &9 &&&9 &9  \\
 &&&9 &9
 }$$
 \medskip

 {\bf Consequences:} We note that the information under (c), (d) below, as well as parts of (e) and (f), can also be obtained via  \cite{KacII} and \cite{Scho}.
\begin{enumerate}[(a)]
\item Let $P =  \bigoplus_{1 \le j \le 5} \la z_r$ be the distinguished projective cover of $\SS_0$.  We apply Theorem \ref{thm-generic-modules} to the skeleton $\S$ to construct a generic minimal projective presentation of the $\bd$-dimensional $\la$-modules.  The $\S$-critical paths in $P$ are $\alpha_2 z_1$, $\alpha_5 z_3$, $\beta_5 z_4$ and $\alpha_8 z_5$.  Choosing elements $x_1, \dots, x_{11} \in \CC$ which are algebraically independent over $\QQ$, we thus obtain the following generic format of a minimal projective presentation:
$G = P/ R(\sigma)$, where $R(\S)$ is generated by the following four elements in $P$:
\begin{gather*}
\alpha_2 z_1, \quad \alpha_5 z_3 - \bigl( x_1 \beta_3 z_2 + x_2 \alpha_5 \alpha_3 z_2 + x_3 \alpha_5 z_4\bigr), \quad \beta_5 z_4 - \bigl(x_4 \beta_5 \alpha_3 z_2  +  x_5 \beta_5 z_3 \bigr) \ \ \ \text{and} \\
\alpha_8 z'_5 - \bigl( x_6 \alpha_7 \beta_3 z_2 + x_7 \gamma_5 \alpha_3 z_2 + x_8 \alpha_7 \alpha_5 \alpha_3 z_2 + x_9 \gamma_5 z_3 + x_{10} \gamma_5 z_4 + x_{11}  \alpha_7 \alpha_5 z_4\bigr).
\end{gather*}
This presentation of $G$ is simplified compared with that suggested by Theorem \ref{thm-generic-modules}; namely, we replace $z_5$ by a top element $z'_5 = z_5$ plus a suitable linear combination of $\alpha_6 \beta_5 \alpha_3 z_2$, $\beta_6 \beta_5 \alpha_3 z_2$, $\alpha_6 \beta_5 z_3$, $\beta_6 \beta_5 z_3$.

\item By Theorem \ref{mainsoc}, the generic socle layering $\SS^*$ of $\modlad$ is
\[\SS^* = (S_2 \oplus S_9^{10}, \, S_7^3 \oplus S_8^5,\,  S_5 \oplus S_6^2,\, S_5^2, \, \,S_3),\]
Moreover, by Corollary \ref{socetc}, the simple $S_2$ is the largest semisimple module generically occurring as a direct summand of the modules in $\laySS$.

\item Generically, the modules with dimension vector $\bd$ decompose into two indecomposable summands, one isomorphic to $S_2$, the other with dimension vector $\bd' = \bd - (0, 1,0 ,\dots,0)$.  Indeed, $G \cong S_2 \oplus G'$, where $G'$ is a generic module for $\Mod \SS'$; here $\SS' = (\SS'_0, \dots,  \SS'_4)$ with $\SS'_0 = S_3 \oplus S_5^2 \oplus S_8$ and $\SS'_l = \SS_l$ for $l \ge 1$.  Letting $P' = \bigoplus_{2 \le j \le 5} \la z_j$ be the distinguished projective cover of $\SS'_0$, we obtain a generic skeleton $\sigma' = \sigma \setminus \{z_1\} \subseteq P'$ for the modules with dimension vector $\bd'$.  A generic minimal projective presentation based on $\S'$ is $G' = P' / (P' \cap R(\S))$.  Using Corollary \ref{cor-gen-endo}, we find that, generically, the modules with dimension vector $\bd'$ have endomorphism rings isomorphic to $\CC$.  Since this guarantees generic indecomposability of the $\bd'$-dimensional modules, our claim is justified. 

\item Generically, the modules with dimension vector $\bd$ (resp., with dimension vector $\bd'$) contain a submodule isomorphic to $\la e_3 \oplus \la e_5$.  This is most easily seen by passing to a different skeleton of $G$.

\item The submodule of $G$ which is generated by (the $R(\S)$-residue classes of) $z_2$, $z_3$, $z_4$ is layer-stably embedded in $G$.   By Corollary \ref{cor-subs-quots-gen-mods}, the module $\la z_2 + \la z_3 + \la z_4$ is therefore generic for its dimension vector.  Again applying Corollary \ref{cor-gen-endo}, one obtains generic indecomposability of the modules with this dimension vector.

\item The modules with dimension vector $(0,0,0,0,2,1, 1, 2, 5)$ have generic skeleta as shown under $z_3$, $z_4$ of $\S$.  Generically, they decompose into two local modules which are unique up to isomorphism.  
\end{enumerate}

\section{Nonhereditary examples}\label{s-nonhered}

In the examples, $Q$ will be an acyclic graph and $\bd$ a dimension vector of $Q$.  Our primary purpose in Example \ref{ex-nonhered-1} is to illustrate the information that results from exploring the generic behavior of the $\bd$-dimensional representations of $Q$ as we vary the allowable Loewy length $L+1$.  In the extreme cases, where $L$ is either the maximal path length, i.e. $L=6$, on one hand, or $L = 1$ on the other, generic direct sum decompositions are already well understood; see \cite{KacII}, \cite{Scho}, and  \cite{BCH}.  As is to be expected, the picture is more complex in the mid-range between these extremes. 

\begin{Example}\label{ex-nonhered-1}  Let $\la_L = \CC Q/ \langle \text{the paths of length}\ L+1 \rangle$, where $Q$ is the quiver
\smallskip
$$\xymatrixrowsep{1pc}\xymatrixcolsep{4pc}
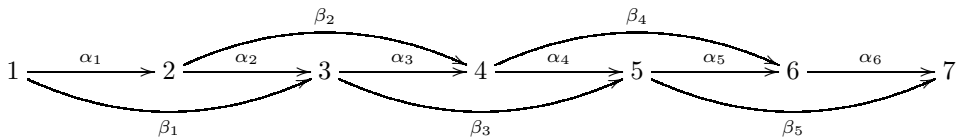
\xymatrix{
1 \ar[r]^{\alpha_1} \ar@/_1.25pc/[rr]_{\beta_1}  &2 \ar[r]^{\alpha_2} \ar@/^1.25pc/[rr]^{\beta_2}  &3 \ar[r]^{\alpha_3} \ar@/_1.25pc/[rr]_{\beta_3}   &4 \ar[r]^{\alpha_4} \ar@/^1.25pc/[rr]^{\beta_4}   &5 \ar[r]^{\alpha_5} \ar@/_1.25pc/[rr]_{\beta_5}   &6 \ar[r]^{\alpha_6}  &7
} \smallskip$$
and $\bd = (1, 1, \dots, 1) \in \NN^7$.

\begin{enumerate}[(a)]
\item Clearly, if $L= 6$, i.e., $\la_L = KQ$, the modules with dimension vector $\bd$ are generically uniserial with radical layering $(S_1, \dots, S_7)$. 

\item For $L = 5$, the variety $\rep_{\bd}(\la_L)$ has precisely $6$ irreducible components, all of them representing generically indecomposable modules.  They are listed in terms of their generic modules which, by Theorem \ref{thm-generic-modules}, are available from the generic radical layerings.  We communicate these modules via their graphs, in Diagram 5.1.1; here the solid edge paths starting at the top represent the chosen skeleton $\S$ in each case, and the edge paths starting at the top and terminating in a broken edge are the $\S$-critical paths, tied in by the relations given in Theorem \ref{thm-generic-modules}.  
\[\xymatrixrowsep{0.75pc}\xymatrixcolsep{0.1pc}
\xymatrix{
1 \edge[dr] &&2 \dashedge[dl]\dashedge@/^1pc/[ddl]  &&&&& &1 \edge[d] \dashedge@/_1pc/[dd] &  &&&&& &1 \edge[d] \dashedge@/_1pc/[dd] &  &&&&& &1 \edge[d] \dashedge@/^1pc/[dd] &  &&&&& &1 \edge[d] \dashedge@/_1pc/[ddl] &  &&&&& &1 \edge[dl] \edge[dr]  \\
&3 \edge[d] \dashedge@/_1pc/[dd] &  &&&&& &2 \edge[d] \dashedge@/^1pc/[dd] &  &&&&& &2 \edge[d] \dashedge@/^1pc/[dd] &  &&&&& &2 \edge[d] \dashedge@/_1pc/[ddl] &  &&&&& &2 \edge[dl] \edge[dr] &  &&&&&2 \edge[dr] &&3 \dashedge[dl] \dashedge@/^1pc/[ddl]  \\
&4 \edge[d] \dashedge@/^1pc/[dd] &  &&&&& &3 \edge[d] \dashedge@/_1pc/[dd] &  &&&&& &3 \edge[d] \dashedge@/_1pc/[ddl] &  &&&&& &3 \edge[dl] \edge[dr] &  &&&&&3 \edge[dr] &&4 \dashedge[dl] \dashedge@/^1pc/[ddl]  &&&&& &4 \edge[d] \dashedge@/_1.25pc/[dd]  \\
&5 \edge[d] \dashedge@/_1pc/[dd] &  &&&&& &4 \edge[d] \dashedge@/_1pc/[ddl] &  &&&&& &4 \edge[dl] \edge[dr] &  &&&&&4 \edge[dr] &&5 \dashedge[dl] \dashedge@/^1pc/[ddl]  &&&&&  &5 \edge[d] \dashedge@/_1pc/[dd] &  &&&&& &5 \edge[d] \dashedge@/^1pc/[dd]  \\
&6 \edge[d] &  &&&&& &5 \edge[dl] \edge[dr] &  &&&&&5 \edge[dr] &&6 \dashedge[dl]  &&&&& &6 \edge[d] &  &&&&&  &6 \edge[d] &  &&&&& &6 \edge[d]  \\
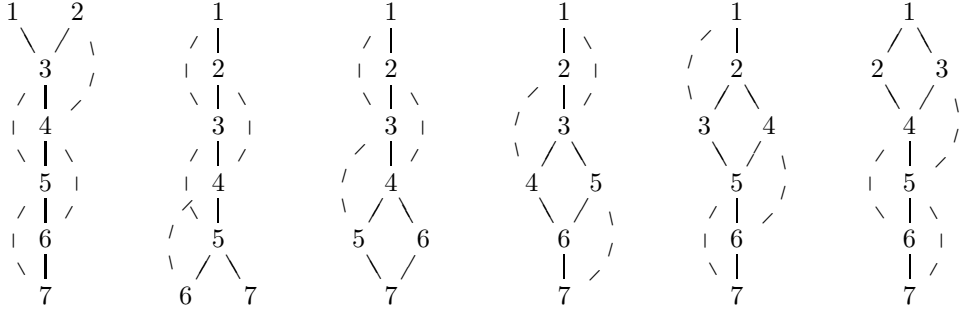
&7 &  &&&&&6 &&7  &&&&& &7 &  &&{\save+<0ex,-5ex> \drop{\txt{{\bf Diagram 5.1.1.} Generic modules for Example \ref{ex-nonhered-1}(b), $L=5$}} \restore} &&& &7 &  &&&&&  &7 &  &&&&& &7
}\]
We add some explanation regarding the leftmost diagram: One of the irreducible components of $\Mod_\bd(\la_5)$ equals the closure of $\laySS$, where $\SS = (S_1\oplus S_2, \, S_3, \, S_4, \, S_5, \, S_6, \, S_7)$. Generically, the modules in this component are of the form $G = (\la z_1 \oplus \la z_2) / C$, where $z_i = e_i$ for $i=1,2$ and $C$ is the $\la$-submodule generated by $\alpha_1 z_1$, $\alpha_2 z_2 - x_1 \beta_1 z_1$, $\beta_2 z_2 - x_2 \alpha_3 \beta_1 z_1$, $\beta_3 \beta_1 z_1 - x_3 \alpha_4 \alpha_3 \beta_1 z_1$, $\beta_4 \alpha_3 \beta_1 z_1 - x_4 \alpha_5 \alpha_4 \alpha_3 \beta_1 z_1$, $\beta_5 \beta_3 \beta_1 z_1 - x_5 \alpha_6 \alpha_5 \alpha_4 \alpha_3 \beta_1 z_1$; here the $x_i \in \CC$ are algebraically independent over $\QQ$ (cf.~Theorem \ref{thm-generic-modules}). 

For each of the semisimple sequences $\SS$ which are generic for the components of $\modlad$, the modules in $\laySS$ have a fine moduli space; see, e.g., \cite[Theorem 4.4, Corollary 4.5]{classifying}.  All of these moduli spaces are $4$-dimensional. To see this for the semisimple sequence $\SS$ discussed in the preceding paragraph, note that, in our presentation of $G$, we may replace $x_1$ by $1$ without affecting the isomorphism type of $G$.

\item The case $L = 3$ is more interesting.  Using the \hyperlink{thm-main}{Main Theorem}, one  finds that the variety $\modlad$ has precisely $28$ irreducible components, $12$ of which encode generically indecomposable modules; the remaining $16$ encode modules which generically split into two indecomposable summands.  The dimensions of the moduli spaces representing the modules with fixed radical layering (existent by \cite[Theorem 4.4]{classifying}) vary among $1$, $2$, $3$ for the different components.  In particular, none of the components contains a dense orbit.  We list $9$ of these components in Diagram 5.1.2 (below), again in terms of graphs of their generic modules.
\begin{figure}[h]
\[\xymatrixrowsep{0.75pc}\xymatrixcolsep{0.1pc}
\xymatrix{
\txt{(A)}  &&1 \edge[dl] \edge[dr]  & &&&&&\txt{(B)} &1 \edge[d] &&2 \edge[d] \dashedge[dll]  &&&&&\txt{(C)} &&1 \edge[dl] \edge[dr] &&&7 \drbl  \\
&2 \edge[d] \dashedge@/_1pc/[dd] &&3 \edge[d] \dashedge[dll]  &&&&& &3 \edge[dr] &&4 \dashedge[dl]  &&&&& &2 \edge[dr] &&3 \dashedge[dl] \dashedge@/^1pc/[ddll] &\bigoplus  \\
&4 \edge[d] &&5 \edge[d] \dashedge[dll]  &&&&& &&5 \edge[dl] \edge[dr]  & &&&&& &&4 \edge[dl] \edge[dr]  \\
&6 &&7  &&&&& &6 &&7  &&&&& &5 &&6  \\
1 \edge[d] &\txt{\hphantom{5}}&3 \edge[d] \dashedge[ddll]  &&&&&&1 \edge[d] &&&3 \edge[dl] \edge[dr]  & &&&&&&1 \edge[d] \dashedge@/_1pc/[dd] &&4 \edge[d] \dashedge@/^1pc/[dd]  \\
2 \edge[d] &&5 \edge[d] \dashedge@/^1pc/[dd]  &&&&&&2 &\bigoplus &4 \edge[dr] &&5 \dashedge[dl] \dashedge@/^/[ddl]  &&&&&&2 \edge[d] &\bigoplus &5 \edge[d] \dashedge@/_1pc/[dd]  \\
4 &&6 \edge[d] &&&&&& &&&6 \edge[d]  & &&&&&&3 &&6 \edge[d]  \\
&&7  &&&&&& &&&7  & &&&&&& &&7  \\
&1 \edge[dl] \edge[dr] &&4 \dashedge@/^/[dddl]  &&&&&&1 \edge[d] \dashedge@/_1pc/[dd] &&&5 \edge[d] \dashedge@/_0.4pc/[dddl]  &&&&&&1 \edge[dr] &&2 \dashedge[dl] \dashedge@/_1pc/[ddll] &6 \dashedge@/^/[dddl]  \\
2  &&3 \edge[dl] & &&&&&&2 \edge[d] \edge[dr] &&&7  &&&&&& &3 \edge[dl] \edge[dr]  \\
&5 \edge[dl] \edge[dr] && &&&&&&3 &4 \edge[dr] && &&&&&&4 &&5 \edge[d]  \\
7 &&6 & &&&&&& &&6{\save+<-0.2pc,-5ex> \drop{\txt{{\bf Diagram 5.1.2.} Some generic modules for Example \ref{ex-nonhered-1}(c), $L=3$}} \restore} && &&&&& &&7
}\]
\end{figure}

Note that the generic radical layering of the component labeled (A) in the diagram is strictly smaller than that of the component labeled (B), while the socle layerings are in reverse relation.  The generic socle layering of (A) is strictly smaller than that of (C), but the generic radical layerings of (A) and (C) are not comparable. 
\end{enumerate}
\end{Example}  

If $J^2 = 0$, the subvarieties $\laySS$ of $\modlad$ constitute a stratification of $\modlad$ in the strict sense, in that all closures of strata are unions of strata.  In fact, in Loewy length $2$ the strata are organized by the equivalence  ``$\laySS \subseteq \overline{\Mod \SShat}$ $\iff$ $(\SShat, \SShat^*) \le (\SS, \SS^*)$"; see \cite[Theorem 3.6]{BCH}.  The nontrivial implication fails badly already for $J^3 = 0$, even when the underlying quiver is acyclic, as the upcoming example demonstrates.

\begin{Example}  Let $\la = KQ/ \langle \text{the paths of length}\ 3 \rangle$, where $Q$ is the quiver
$$\xymatrixrowsep{1.5pc}\xymatrixcolsep{3pc}
\xymatrix{
1 \ar[r]^{\alpha_1} &2 \ar[r]^{\alpha_2} \ar[d]_{\beta} &3  \\
4 \ar[r]^{\alpha_4} &5 \ar[r]^{\alpha_5} &6
}$$
 Again take $\bd = (1, \dots, 1)$.  Then $\modlad$ has precisely $3$ irreducible components, one of which is the closure  of $\Mod \SShat$, where $\SShat = (S_1 \oplus S_4,\, S_2 \oplus S_5,\, S_3 \oplus S_6)$.  All modules in this component are annihilated by $\beta$.  If $\SS = (S_1 \oplus S_2 \oplus S_4 \oplus S_6,\, S_3 \oplus S_5, 0)$, and if the generic socle layerings of the modules in $\Mod \SShat$ and $\laySS$ are denoted by $\SShat^*$ and $\SS^*$, respectively, we find $(\SShat, \SShat^*) < (\SS, \SS^*)$.  On the other hand, $\laySS \not\subseteq \overline{\Mod \SShat}$, since, generically, the modules in $\laySS$ are not annihilated by $\beta$.  However, observe that $\laySS\, \cap\, \overline{\Mod \SShat} \ne \varnothing$; indeed, the direct sum $S_1 \oplus \bigl( \la e_2/ (\la \alpha_2 + \la \beta) \bigr) \oplus \bigl(\la e_4/ \la \alpha_4 \bigr) \oplus S_6$ belongs to the intersection.  

The components containing the irreducible variety $\laySS$ are determined by the generic radical layerings $(S_1 \oplus S_2 \oplus S_4,\, S_3 \oplus S_5,\,  S_6)$ and $(S_1 \oplus S_4 \oplus S_6, S_2, S_3 \oplus S_5)$.   
\end{Example}   
 
 %%%%%%%%%%%
\section{Proof of the Main Theorem}\label{s-proof-main-thm}

Throughout this section, we suppose that $\la = KQ/ \langle \text{the paths of length} \ L+1\rangle$, based on a quiver $Q = (Q_0,Q_1)$. It is only after Lemma \ref{newL6.1} that we assume $Q$ to be acyclic.
In particular, the following ancillary result holds without any additional assumptions on $\la$.  As a subsidiary, we will use the semisimple subalgebra $KQ_0$ of $\la$ generated by the paths of length zero.

\begin{Lemma}  \label{newL6.1}
Let $P \in \lamod$ be projective, and $0 \le l \le L$. Then:

{\rm(a)} $J^l P$ is a projective $(\la/J^{L+1-l})$-module.

{\rm(b)} If $V$, $W$ are $KQ_0$-submodules of $J^l P$ with $J^lP = V\oplus W\oplus J^{l+1}P$, then 
$$J^{l+1} P = (Q_1\cdot V) \oplus (Q_1\cdot W) \oplus J^{l+2} P.$$

{\rm (c)} If $C$ is a $\la$-submodule of $P$, then $JC \cap J^{l+1} P = Q_1 \cdot (C\cap J^l P)$.
\end{Lemma}

\begin{proof} Suppose $P = \la z_1 \oplus \cdots \oplus \la z_s$ where the $z_r$ are top elements of $P$ normed by vertices $e(r) \in Q_0$ respectively. Then the following is a $K$-basis for $J^l P$:
$$\bigsqcup_{1\le r\le s} \{ p\, z_r \mid p = p e(r) \; \text{is a path in} \; Q \; \text{with} \; l \le \text{length}(p) \le L \}.$$
In light of this fact, the proof of the lemma is an easy exercise.
\end{proof}

\begin{proof}[\textbf{Proof of Main Theorem}]

First we note that the implication ``(2)$\implies$(1)" only requires that $\la$ be a truncated path algebra (see \cite[Theorem 3.1]{irredcompI}). The final statements of the theorem are clear.

Now suppose $\la = KQ/ \langle \text{the paths of length} \ L+1\rangle$, where $Q$ is an \emph{acyclic} quiver, and $\bd$ is a dimension vector.

``(1)$\implies$(2)":  Note that the minimal elements in $\radsocbd$ coincide with the minimal elements in 
$$\radsocbd^* := \{ (\SS,\SS^*) \mid \SS \; \text{realizable, and} \; \SS^* \; \text{is the generic socle layering of the modules in} \; \laySS \}.$$
Suppose $\SS$ is a semisimple sequence with dimension vector $\bd$ which does not arise as the first entry of any minimal pair in $\radsocbd$. We will show that $\overline{\laySS}$ fails to be an irreducible component of $\modlad$. Denoting by $\SS^*$ the generic socle layering of the modules in $\laySS$, we find that the pair $(\SS,\SS^*)$ is then a non-minimal element of $\radsocbd^*$. By Observation \ref{com-gen-modules}, we do not lose generality in passing to a suitable extension field of $K$ which has infinite transcendence degree over its prime field.  Hence we may assume that for any realizable semisimple seqence $\SS'$, there is a generic $\la$-module for $\laySS'$.  

Our strategy is to construct, for each generic module $G$ for $\laySS$, a realizable semisimple sequence $\SStilde$ with dimension vector $\bd$ which is different from $\SS = \SS(G)$ such that $G$ belongs to the closure of $\Mod \SStilde$ in $\modlad$.  Since the generic modules form a dense subset of $\laySS$ \cite[Corollary 4.5]{BHT} and $\Seq(\bd)$ is finite,  this will imply 
$\overline{\laySS} \subsetneqq \overline{ \Mod \SStilde}$ for a suitable $\SStilde$, allowing us to conclude that $\overline{\laySS}$ is not an irreducible component of $\modlad$. 

Fix a generic module $G$ for $\laySS$.  In particular, this implies that $\SS^*$ is the socle layering of $G$.

We will first pin down a suitable short exact sequence
$$ 0 \rightarrow A \rightarrow G \rightarrow B \rightarrow 0$$
representing a class $\eta \in \Ext_\la^1(B,A)$ say.  Then we will construct another class $\xi \in \Ext_\la^1(B,A)$ and consider the one-parameter family of extensions
$$0 \rightarrow A \rightarrow G_t \rightarrow B \rightarrow 0$$
corresponding to $\eta + t \xi$, $t \in K$.  Our construction will be to the effect that $G_0 \cong G$, while $\SS(G_t) \ne \SS(G)$ for general $t$.  The sequence $\SStilde$ we seek will then be among the $\SS(G_t)$.

Our assumption on $\SS$ provides us with a pair $(\SShat, \SShat^*) \in \radsocbd^*$ which is strictly smaller than $(\SS, \SS^*)$ under the componentwise dominance order. Say 
$(\SShat, \SShat^*)  = (\SS(\Ghat), \SS^*(\Ghat))$  for some $\la$-module $\Ghat$.  It is harmless to assume that $\Ghat$ is a generic module for the radical layering $\SShat$.
Clearly, $\SShat < \SS$, since  from $\SShat = \SS$ we would deduce $\SS^*(G) = \min\{\SS^*(N) \mid N \in \overline{\laySS} \}  =  \SS^*(\Ghat)$.

In light of the equality \, $\bigoplus_{0 \le l \le L} \SS_l  = \bigoplus_{0 \le l \le L}  \SShat_l$,\, the inclusions \, $\bigoplus_{l \le j} \SS_l \supseteq  \bigoplus_{l \le j} \SShat_l$\, for $j \le L$ imply \, $\bigoplus_{l \ge j} \SS_l \subseteq  \bigoplus_{l \ge j} \SShat_l\,$ for $j \le L$.  In particular, $\SS_L \subseteq \SShat_L$.  If $\SS_L = \SShat_L$, then $\SS_{L-1} \subseteq \SShat_{L-1}$, etc.  This means that there is an index $v$ with the property that $\SS_v \subsetneqq \SShat_v$. 

We choose $\tau \in \{0, \dots, L\}$ minimal with respect to $\SShat_\tau \not\subseteq \SS_\tau$.  Then $\tau \ge 1$ because $\SShat_0 \subseteq \SS_0$ in view of  $\,\SShat < \SS$.  Pick $k \in \{1, \dots, n\}$ such that $\dim_K (e_k \SShat_\tau) >  \dim_K (e_k \SS_\tau)$.  

\begin{Claim}\label{claim-1}
There is an element $a \in e_kG$ such that $a \notin J^\tau G$ but $J^{L - \tau +1}a = 0$.
\end{Claim}

\begin{proof}  Note that the annihilator of $J^{L-\tau+1}$ in $G$ coincides with $\soc_{L-\tau}(G)$.  If the claim were false, we would thus obtain $e_k\soc_{L-\tau}(G) \subseteq e_kJ^\tau G$.  Since always $J^\tau G \subseteq \soc_{L-\tau}(G)$ (analogously for $\Ghat$), this would amount to $e_k \soc_{L-\tau}(G) = e_k J^\tau G$.  On the other hand, by the construction of $\tau$ and $k$,  
\begin{multline*}
\dim_K(e_kJ^\tau G) - \dim_K(e_k J^{\tau+1}G) = \dim_K(e_k \SS_\tau) < \\
\dim_K(e_k \SShat_\tau) = \dim_K(e_k J^\tau \Ghat) - \dim_K(e_k J^{\tau+1}\Ghat), 
\end{multline*}
and, combining with our assumption, we derive 
\begin{multline*}
\dim_K(e_k \soc_{L-\tau}G)) \leq \\
\dim_K(e_k J^\tau G) + [\dim_K(e_k J^{\tau+1}\Ghat) - \dim_K(e_k J^{\tau+1}G)] < \dim_K(e_k \soc_{L-\tau}(\Ghat)); 
\end{multline*}
keep in mind that $\dim_K(e_k J^{\tau+1} G) \leq \dim_K(e_k J^{\tau+1} \Ghat)$ due to $\SShat \le \SS$.  However, $\SS^\ast(\Ghat) \leq \SS^*(G)$ implies
$$\dim_K(e_k \soc_{L-\tau}(\Ghat)) \leq \dim_K(e_k \soc_{L-\tau}(G)),$$
a contradiction.  
\end{proof}

\begin{Claim}\label{claim-2}
Let $A = \Lambda a \subseteq G$, and set $B = G/A$.  Then the semisimple sequence $\bigl(\SS_{\tau-1}(B),\SS_\tau(B) \oplus S_k \bigr)$ is realizable.
\end{Claim}

\begin{proof}
We repeatedly use the realizability criterion \ref{thm-realizability-criterion}.  Clearly, it suffices to prove that, for some submodule $\SS'$ of the semisimple module $\SS_{\tau - 1}(B)$, the sequence $\bigl( \SS', e_k \SS_\tau(B) \oplus S_k \bigr)$ is realizable.
Recall that we identify the vertices of $Q$ with the corresponding primitive idempotents of $\la$.  Hence it makes sense to consider the sum $e$ of the starting vertices of the paths of positive length ending in $e_k$; clearly, $e$ is an idempotent in $\la$.  We will show realizability of the semisimple sequence $\bigl( e\, \SS_{\tau - 1}(B) , e_k \SS_\tau(B) \oplus S_k \bigr)$, which will cover our claim.  

Acyclicity of $Q$ yields $e \la e_i = 0$\, for any vertex $e_i$ with $e_i \la e_k \ne 0$.  In light of $a =  e_k a$, we deduce that $e J^l B \cong e J^l G$ canonically (as $K$-spaces) for all $l \le L$, whence $e\, \SS_l(B) = e\, \SS_l(G)$.  Moreover, we find $e_k \SS_\tau(B) = e_k \SS_\tau(G)$ because $a \notin J^\tau G$ and $e_k JA = 0$.   

Due to our choice of $\tau$ and $k$, we have $e\, \SS_{\tau - 1} (\Ghat) \subseteq e\, \SS_{\tau - 1}(G)$, while $e_k \SS_\tau(\Ghat) \supsetneqq  e_k \SS_\tau(G)$.  Therefore realizability of $\bigl( e\, \SS_{\tau - 1} (\Ghat),  e_k \SS_\tau(\Ghat) \bigr)$ implies realizability of 
$$\bigl( e\, \SS_{\tau - 1}(G),\, e_k \SS_\tau(G) \oplus S_k \bigr) =  \bigl( e\, \SS_{\tau - 1}(B), \, e_k \SS_\tau(B) \oplus S_k \bigr)$$
as required. 
\end{proof}

\begin{Claim}\label{claim-3}
Let $\chi \in \Ext^1_\la(B, A/JA)$.  If the canonical image of $\chi$ in $\Ext^1_\la(J^\tau B, A/JA)$ is zero, then $\chi$ lifts to a class in $\Ext^1_\la(B,A)$.
\end{Claim}

\begin{proof} Suppose $ \chi $ is represented by an extension 
$$0 \to A/JA \overset u \longrightarrow X \overset v \longrightarrow B \to 0 \, ,$$
and assume $ \chi $ maps to zero in $\Ext^1_\la(J^\tau B, A/JA)$.  If $ X_\tau = v^{-1}(J^\tau B)$, then the class of the exact sequence
\[ 0 \to A/JA \to X_\tau \to J^\tau B \to 0 \]
is the image of $ \chi $ in $ \Ext^1_\la(J^\tau B, A/JA) $.  This class being zero, there exists a $ \la $-direct summand $ Y_\tau$ of $X_\tau $ which maps isomorphically onto $ J^\tau B $ under $v$.  Let $ KQ_0 $ again denote the semisimple subalgebra of $ \la $ generated by the paths of length zero.  We extend the $ \la $-module splitting $ X_\tau = u(A/JA) \oplus Y_\tau $ to a $ K Q_0 $-module splitting $ X = u(A/JA) \oplus B' $ of $ \chi $ in the sense that $ Y_\tau \subseteq B' $.

Consider the $ K Q_0 $-module $ M = A \oplus B' $.  Referring to the distinguished element $ a \in A $, we decompose the $KQ_0$-module $A$ in the form $ A = Ka \oplus JA $.  Clearly, we may identify $ Ka $ with $ A/JA $ via the natural map $ Ka \to A \to A/JA $.  As a consequence, we obtain a natural isomorphism $\phi : JA \oplus X \rightarrow M$ of $ K Q_0 $-modules which restricts to the identity on $JA \oplus B'$.  For $ \alpha \in Q_1 $, let $ f_\alpha\colon A \to A $ and $ g_\alpha: B' \to X $ denote the action and restricted action of $ \alpha $ on $ A $ and $ B' $, respectively.    We now define a $ K Q $-module structure on $ M $:  For $ a \in A $ and $ b \in B' $, we set
\[ \alpha ( a + b) = f_\alpha(a) + \phi g_\alpha(b) \quad \forall\; \alpha \in Q_1 \, . \]
Observe that the natural inclusion $ A \to M $ is a homomorphism of $ K Q $-modules and that $ Y_\tau \subseteq M $ is a $ KQ$-submodule.  We thus arrive at the following commutative diagram with exact rows in $KQ\text{-mod}$, 
$$\xymatrixrowsep{2pc}\xymatrixcolsep{2pc}
\xymatrix{
0 \ar[r] &A/JA \ar[r]^-{u} &X \ar[r]^-{v} &B \ar[r] &0  \\
0 \ar[r] &A \ar[u]^{\pi} \ar[r]^-{u'} &M \ar[u]_{q} \ar[r]^-{v'} &B \ar@{=}[u] \ar[r] &0
}$$
where $\pi$ and $u'$ are the canonical quotient and injection maps, while $q|_A$ equals $u\pi$ and $q|_{B'}$ is the inclusion map. The left hand square of this diagram is a pushout (see, e.g., \cite[Lemma 7.18]{Rot}). Furthermore, the fact that $ X_\tau $ decomposes as $ u(A/J) \oplus Y_\tau $ implies that $q^{-1}(X_\tau) $ decomposes as $ A \oplus Y_\tau $ in $KQ\text{-mod}$.

It remains to be checked that, under this $ K Q $-module structure, $ M $ is annihilated by all paths of length $ L+1 $.  Suppose $ p $ is such a path.  We factor $ p $ in the form $ p = p_2 p_1 $, where $ p_1 $ and $ p_2$ are paths of lengths $ \tau $ and $ L+1 - \tau $, respectively.  It is clear that $ p $ annihlates $ A $.  So let $ b \in B' $.  Then $v'(p_1b)$ lies in $ J^\tau B $, whence, by the preceding paragraph, $ p_1 b $ lies in the $ K Q $-submodule $ A \oplus Y_\tau $ of $M$.  Since both $ A $ and $ Y_\tau$ are annihilated by paths of length $ L+1-\tau $ (keep in mind that $Y_\tau \cong J^\tau B$), we find that $ p_2 p_1 b = 0 $ as required.

Therefore the bottom row of the above diagram  represents the postulated lift of $\chi$. 
\end{proof}

The next step is based on Claim \ref{claim-2}.  It will provide us with a suitable nontrivial element $\chi \in \Ext_\la^1(B, A/JA)$ satisfying the hypothesis of Claim \ref{claim-3}.

\begin{Claim}\label{claim-4}
The canonical map 
$\bE: \Ext_\la^1(B/J^\tau B, S_k) \longrightarrow \Ext_\la^1(J^{\tau-1}B, S_k)$
is nonzero.
\end{Claim}

\begin{proof} Let $\la_0$ be the truncated path algebra $\la/J^{\tau+1}$, set $B_0 = B/ J^{\tau+1} B \in \la_0\text{-mod}$, and let $P_0$ be a $\la_0$-projective cover of $B_0$ with kernel $C$.  It is clearly innocuous to identify $B_0$ with $P_0/C$.  Our first goal is to construct a $\la_0$-module $B'$ such that $J^\tau B'$ contains a simple submodule $A' \cong S_k$ with $B'/A' \cong B_0$. We will construct $B'$ in the form $P_0/C'$ where $C'$ is a submodule of $C$ with $C/ C' \cong S_k$.

Consider a $KQ_0$-module decomposition
$J^{\tau-1} P_0 = V \oplus W \oplus J^\tau P_0$
such that
$$W+ J^\tau P_0 = (C \cap J^{\tau-1} P_0) + J^\tau P_0 \, .$$
This yields a $KQ_0$-isomorphism $V \cong \SS_{\tau-1}(B_0)$. By Lemma \ref{newL6.1}, applied to the truncated path algebra $\la_0$ and $P = P_0$, we find that $V \oplus (Q_1\cdot V)$ is a projective $(\la/J^2)$-module, and we obtain a $\la_0$-module decomposition
\begin{equation}  \label{JtauP0}
J^\tau P_0 = (Q_1\cdot V) \oplus (Q_1\cdot W);
\end{equation} 
indeed, by construction,  $J^{\tau+1} P_0 = 0$. Due to realizability of $ ( \SS_{\tau-1}(B_0), \SS_\tau(B_0) \oplus S_k ) $ (see Claim \ref{claim-2}), the kernel of the canonical map $Q_1\cdot V \rightarrow J^\tau B_0$ contains a copy $S$ of the simple module $S_k$; in particular, $S$ is a simple $\la_0$-submodule of $C\cap J^\tau P_0$. Showing that $S \cap JC = 0$ will thus provide us with a $\la_0$-submodule $C'$ of $C$ as specified above.

Suppose that $S\cap JC \ne 0$. Then 
$$S \subseteq JC \cap J^\tau P_0 = Q_1\cdot (C\cap J^{\tau-1} P_0) = Q_1\cdot W$$
by Lemma \ref{newL6.1} and our choice of $W$. But, in light of \eqref{JtauP0}, this containment is not compatible with $S \subseteq Q_1\cdot V$.

We now view $B' = P_0/C'$ as a $\la$-module. 
 By construction, we have $B'/J^\tau B' \cong B /J^\tau B$, and in light of $J^{\tau+1} B' = 0$, we obtain $J^\tau B' = (J^\tau B/ J^{\tau+1} B) \oplus A'$.
Thus we have an exact sequence
\begin{equation}\label{eq-starting-seq}
0 \rightarrow \SS_\tau(B) \oplus S_k \rightarrow B' \rightarrow B/J^\tau B \rightarrow 0.
\end{equation}

Pulling \eqref{eq-starting-seq}
back along the inclusion $\, \iota: \SS_{\tau-1}(B) \rightarrow B / J^{\tau} B\, $, we obtain the induced exact sequence
\begin{equation}\label{eq-induced}
\xymatrixrowsep{2.5pc}\xymatrixcolsep{1.5pc}
\xymatrix{
0 \ar[r] &\SS_{\tau}(B) \oplus S_k \ar[r]^-{u} &J^{\tau-1}B' \ar[r]^-{v} &\SS_{\tau-1}(B) \ar[r] &0
}
\end{equation}
where $\Im(u) = J^\tau B'$.  This equality implies that the map $\Hom_\la(u, S_k)$ is zero, whence the connecting homomorphism
\[ \delta: \Hom_\la\bigl( \SS_{\tau}(B) \oplus S_k,\, S_k \bigr) \to \Ext^1_\la( \SS_{\tau-1}(B), \,S_k) \]
is injective.  Letting $\, p: \SS_{\tau}(B) \oplus S_k \rightarrow S_k$ be the canonical projection, we infer that the exact sequence 
\begin{equation}  \label{push-p}
\xymatrixrowsep{2.5pc}\xymatrixcolsep{1.5pc}
\xymatrix{
0 \ar[r] &S_k \ar[r]^-{u'} &M' \ar[r]^-{v'} &\SS_{\tau-1}(B) \ar[r] &0
}
\end{equation}
resulting from pushing out \eqref{eq-induced} along $p$  is non-split. 
 
 We claim that the canonical epimorphism $\pi: J^{\tau - 1} B \rightarrow \SS_{\tau - 1}(B)$ does not factor through $v'$. Assume to the contrary that $\pi = v' f$ for some $f \in \Hom_\la(J^{\tau-1} B, M')$. Since $J^2 M' = 0$ (see \eqref{push-p}), $J^{\tau+1} B$ is contained in the kernels of $f$ and $\pi$, whence we obtain a factorization $(\overline{\pi},0) = v' (\overline{f},0)$ where
 $$(\overline{\pi},0) : (J^{\tau-1} B / J^{\tau+1} B) \oplus S_k \rightarrow \SS_{\tau-1}(B) \qquad\text{and}\qquad (\overline{f},0) : (J^{\tau-1} B / J^{\tau+1} B) \oplus S_k \rightarrow M'$$
and $\overline{\pi}$, $\overline{f}$ are the maps induced from $\pi$, $f$. Since $\ker (\overline{\pi},0) = \SS_\tau(B) \oplus S_k$, this yields a commutative diagram
$$\xymatrixrowsep{2pc}\xymatrixcolsep{2.5pc}
\xymatrix{
0 \dropleft{\eqref{push-p}} \ar[r] &S_k \ar[r]^-{u'} &M' \ar[r]^-{v'} &\SS_{\tau-1}(B) \ar[r] &0  \\
0 \ar[r] &\SS_\tau(B) \oplus S_k \ar[u]^{g} \ar[r] &(J^{\tau-1} B / J^{\tau+1} B) \oplus S_k \ar[u]_{(\overline{f},0)} \ar[r]^-{(\overline{\pi},0)} &\SS_{\tau-1}(B) \ar@{=}[u] \ar[r] &0
}$$
That this is a pushout diagram may be checked directly, or else obtained from \cite[Lemma 7.18]{Rot} for instance.
Therefore, the exact sequence \eqref{push-p} also represents $\delta(g)$. Since the right-hand direct summand $S_k$ of $\SS_\tau(B) \oplus S_k$ is contained in $\ker g$, we see that $g \ne p$. But in light of $\delta(g) = \delta(p)$, this contradicts the injectivity of $\delta$.
Therefore $\pi$ fails to factor through $v'$, as claimed. 

We conclude that the pullback of   
\eqref{push-p} along $\pi$ yields a non-split exact sequence 
\begin{equation}  \label{pull-pi}
\xymatrixrowsep{2.5pc}\xymatrixcolsep{2pc}
\xymatrix{
0 \ar[r] &S_k \ar[r] &M'' \ar[r] &J^{\tau-1} B \ar[r] &0 \, .
}
\end{equation}   
 
 Performing the pullback and pushout operations that led from \eqref{eq-starting-seq} to \eqref{push-p} in reverse order leads to an extension equivalent to \eqref{push-p}.  Let $\psi \in \Ext^1_\la(B/J^\tau B, S_k)$ be the extension class obtained from \eqref{eq-starting-seq} by pushing out along $p$.  Then $\bE(\psi) \in \Ext^1(J^{\tau - 1}B, S_k)$ is the result of first pulling back along $\iota$, which yields the class of \eqref{push-p}, and then pulling \eqref{push-p} back along $\pi$ to obtain the class of \eqref{pull-pi} in $\Ext^1_\la(J^{\tau - 1} B, S_k)$.  Thus $\psi$ does not belong to the kernel of $\bE$. 
\end{proof}

By Claim \ref{claim-4}, we may choose a class $\chi' \in \Ext_\la^1(B/J^\tau B,A/J A)$ with nonzero image in $\Ext_\la^1(J^{\tau-1} B, A/JA)$.  We set $\chi = \Ext_\la^1(\pi, A/JA)(\chi')$, where $\pi:  B \rightarrow B/ J^\tau B$ is canonical.  Since the image of $\Ext_\la^1(\pi, A/J A)$ in $\Ext_\la^1(B, A/JA)$ is the kernel of the natural map 
$$\Ext_\la^1(B, A/JA) \rightarrow \Ext_\la^1(J^\tau B, A/JA),$$
Claim \ref{claim-3} guarantees that $\chi$ lifts to a class $\xi$ in $\Ext_\la^1(B,A)$.  The image of $\xi$ in $\Ext_\la^1(J^{\tau-1} B, A/J A)$ is nonzero by construction.  Hence, for general $t$, the image of the class $\eta + t \xi$ in $\Ext_\la^1(J^{\tau-1}B, A/JA)$ is nonzero, where $\eta$ denotes the class of $0 \to A \to G \to B \to 0$ in $\Ext_\la^1(B,A)$.  Consider a one-parameter family of extensions 
\[ 0 \longrightarrow A \,  \stackrel{f_t}{\longrightarrow}\, G_t \,  \stackrel{g_t}{\longrightarrow}\, B \longrightarrow 0 \]
corresponding to the $\eta + t \xi$, and note that $G_0 = G$.  It is straightforward to translate the family $(G_t)_{t \in K}$ into a curve $\AA^1 \rightarrow \modlad$.  

\begin{Claim}\label{claim-5}
For general $t \in K$ we have $\dim_K(J^\tau G) < \dim_K(J^\tau G_t)$; in particular, $\SS(G) \ne \SS(G_t)$.
\end{Claim}

\begin{proof} For any $t \in K$, the inverse image of $J^\tau B$ under $g_t$ is $f_t(A) + J^\tau G_t$.  Hence $\dim_K(f_t(A) + J^\tau G_t)$ is constant in $t$.  Let $t$ be such that $\eta + t \xi$ has nonzero image $\zeta_t$ in $\Ext_\la^1(J^{\tau-1}B, A/JA)$, and consider the extension
$$0 \rightarrow A/J A \rightarrow\, (f_t(A) + J^{\tau-1} G_t) / J f_t(A) \, \rightarrow J^{\tau-1} B \rightarrow 0$$
representing $\zeta_t$.  Non-splitness forces the simple module $A/JA$ into the radical of the middle term, that is,  $f_t(A) \subseteq J f_t(A) + J^\tau G_t$.  It follows that $f_t(A)  \subseteq J^\tau G_t$, whence $\dim_K( J^\tau G_t) = \dim_K( f_t(A)+ J^\tau G_t)$.  On the other hand, $J^\tau G\, \subsetneqq \, A + J^\tau G$ by Claim 1.  Thus $\dim_K(J^\tau G) < \dim_K(J^\tau G_t)$ as claimed.
\end{proof}

We now choose $\SStilde$ with the property that $\SS(G_t) = \SStilde$ for infinitely many $t$. Then $\SStilde \ne \SS(G)$ and $G \in \overline{\laySStilde}$, as desired.

 This completes the proof of ``(1)$\implies$(2)''. 
\end{proof}

%%%%%%%%%%%%%%%%%

%%%%%%%%%%%%%%%%%%%%%%%%
\end{document}